\input{ifpdf.sty}
\pdffalse

\makeatletter

\newdimen\paperwidth
\newdimen\paperheight

\def\papersize#1#2{\let\p@persize\relax\paperwidth#1\paperheight#2}

\def\Afour{\papersize{210truemm}{297truemm}}

\let\p@persize\Afour

\let\onesidestyle\@twosidefalse
\let\twosidestyle\@twosidetrue

\def\margins{\@ifnextchar[{\@margins}{\@margins[\z@]}}

\def\@margins[#1]#2#3{
  \p@persize\dimen0 #3\dimen0 .5\dimen0\normalsize%
  \oddsidemargin-1truein\advance\oddsidemargin#2%
  \evensidemargin-1truein\advance\evensidemargin#2%
  \topmargin-1truein\advance\topmargin\dimen0\headsep\dimen0\footskip\dimen0%
  \textwidth\paperwidth\advance\textwidth-#2\advance\textwidth-#2%
  \textheight\paperheight\advance\textheight-#3\advance\textheight-#3%
  \headheight\baselineskip\advance\topmargin-.5\baselineskip%
  \advance\headsep-.5\baselineskip%
  \footheight\baselineskip
  \advance\textwidth-#1\advance\oddsidemargin#1
  \if@twoside\def\@themargin%
    {\ifodd\count\z@\oddsidemargin\else\evensidemargin\fi}\fi}

\def\headlinesep#1{\advance\topmargin\headsep\advance\topmargin -#1
  \advance\topmargin.5\baselineskip\headsep #1\advance\headsep-.5\baselineskip}

\def\headline{\if@twoside\let\n@xt\h@dlin@\else\let\n@xt\h@@dlin@\fi\n@xt}
  
\def\h@dlin@#1#2{%
  \def\@oddhead{%
    {{\leftskip\z@\rightskip\z@\noindent\normalsize#1}}}
  \def\@evenhead{%
    {{\leftskip\z@\rightskip\z@\noindent\normalsize#2}}}}

\def\h@@dlin@#1{%
  \def\@oddhead{{{\leftskip\z@\rightskip\z@\noindent\normalsize#1}}}}

\def\footline{\if@twoside\let\n@xt\f@tlin@\else\let\n@xt\f@@tlin@\fi\n@xt}

\def\f@tlin@#1#2{%
  \def\@oddfoot{%
    {{\leftskip\z@\rightskip\z@\noindent\normalsize#1}}}
  \def\@evenfoot{%
    {{\leftskip\z@\rightskip\z@\noindent\normalsize#2}}}}

\def\f@@tlin@#1{%
  \def\@oddfoot{{{\leftskip\z@\rightskip\z@\noindent\normalsize#1}}}}

\def\normalpage{\global\@specialpagefalse}

\def\ft{\@ifnextchar[{\ft@s}{\ft@}}
\def\ft@{\ft@@@s[\f@size]}
\def\ft@s[{\@ifnextchar{a}{\ft@sz[}{\ft@@s[}}
\def\ft@@s[{\@ifnextchar{s}{\ft@sz[}{\ft@@@s[}}
\def\ft@@@s[#1]{\ft@sz[at #1pt]}
\def\ft@sz[#1]#2{\font\fonttemp=#2 #1\fonttemp\ignorespaces}

\def\pbf{\series\bfdefault\selectfont\boldmath}

\def\@@bold{bold}

\def\widebar{\ifx\math@version\@@bold
  \let\@widebar\@@@widebar\else\let\@widebar\@@widebar\fi\@widebar}

\def\@@widebar#1{\text{\setbox15\hbox{$#1$}%
  \dimen15 0.45\wd15\advance\dimen15 0.15\ht15%
  \dimen16\ht15\advance\dimen16 0.00em\advance\dimen16 0.3ex%
  \dimen17 0.65\wd15\advance\dimen17 0.05\ht15\advance\dimen17 0.1ex%
  \dimen18 0.035em\advance\dimen18 0.00ex
  \put[\dimen15,\dimen16][c]{\vrule depth 0pt height \dimen18 width \dimen17}}#1}

\def\@@@widebar#1{\text{\setbox15\hbox{$#1$}%
  \dimen15 0.45\wd15\advance\dimen15 0.15\ht15%
  \dimen16\ht15\advance\dimen16 0.00em\advance\dimen16 0.26ex%
  \dimen17 0.65\wd15\advance\dimen17 0.05\ht15\advance\dimen17 0.1ex%
  \dimen18 0.05em\advance\dimen18 0.00ex
  \put[\dimen15,\dimen16][c]{\vrule depth 0pt height \dimen18 width \dimen17}}#1}

\def\smallsquare{\raise-.065em\hbox{$\Box$}}
\def\smallblacksquare{%
  \kern.3ex\vrule depth-.03ex height1.27ex width1.15ex \kern-1.45ex \smallsquare}

\def\smallcircc{\mathop{\mkern3.5mu\text{\raise.58ex\hbox{\ft{lcircle10}a}}}}
\def\varemptyset{{\text{\raise.21ex\hbox{$\not$}}\mkern.15mu\mathrm{O}\mkern.15mu}}

\def\put{\@ifnextchar[{\@put}{\@@rput[\z@,\z@][r]}}
\def\@put[#1]{\@ifnextchar[{\@@put[#1]}{\@@@@@put[#1]}}
\def\@@put[#1][{\@ifnextchar{l}{\@@lput[#1][}{\@@@put[#1][}}
\def\@@@put[#1][{\@ifnextchar{c}{\@@cput[#1][}{\@@@@put[#1][}}
\def\@@@@put[#1][{\@ifnextchar{r}{\@@rput[#1][}{\relax}}
\def\@@@@@put[{\@ifnextchar{l}{\@@lput[\z@,\z@][}{\@@@@@@put[}}
\def\@@@@@@put[{\@ifnextchar{c}{\@@cput[\z@,\z@][}{\@@@@@@@put[}}
\def\@@@@@@@put[{\@ifnextchar{r}{\@@rput[\z@,\z@][}{\@@@@@@@@put[}}
\def\@@@@@@@@put[#1]{\@@rput[#1][r]}

\let\hm@d@\leavevmode

\long\def\@@lput[#1,#2][l]#3{\setbox0\hbox{#3}\hm@d@\raise#2\hbox to\z@{\dimen0 #1%
  \advance\dimen0-\wd0\kern\dimen0\dp0\z@\ht0\z@\wd0\z@\box0\hss}\ignorespaces}
\long\def\@@cput[#1,#2][c]#3{\setbox0\hbox{#3}\hm@d@\raise#2\hbox to\z@{\dimen0 #1%
  \advance\dimen0-.5\wd0\kern\dimen0\dp0\z@\ht0\z@\wd0\z@\box0\hss}\ignorespaces}
\long\def\@@rput[#1,#2][r]#3{\setbox0\hbox{\kern#1\raise#2\hbox{#3}}%
  \dp0\z@\ht0\z@\wd0\z@\hm@d@\box0\ignorespaces}

\def\flbox{\@ifnextchar[{\@flbox}{\@@rflbox[\z@,\z@][r]}}
\def\@flbox[#1]{\@ifnextchar[{\@@flbox[#1]}{\@@@@@flbox[#1]}}
\def\@@flbox[#1][{\@ifnextchar{l}{\@@lflbox[#1][}{\@@@flbox[#1][}}
\def\@@@flbox[#1][{\@ifnextchar{c}{\@@cflbox[#1][}{\@@@@flbox[#1][}}
\def\@@@@flbox[#1][{\@ifnextchar{r}{\@@rflbox[#1][}{\relax}}
\def\@@@@@flbox[{\@ifnextchar{l}{\@@lflbox[\z@,\z@][}{\@@@@@@flbox[}}
\def\@@@@@@flbox[{\@ifnextchar{c}{\@@cflbox[\z@,\z@][}{\@@@@@@@flbox[}}
\def\@@@@@@@flbox[{\@ifnextchar{r}{\@@rflbox[\z@,\z@][}{\@@@@@@@@flbox[}}
\def\@@@@@@@@flbox[#1]{\@@rflbox[#1][r]}
\long\def\@@lflbox[#1,#2][l]#3{\@@lput[#1,#2][l]{%
  \vtop{\leftskip\z@\parindent\z@\raggedleft\hm@d@#3}}}
\long\def\@@cflbox[#1,#2][c]#3{\@@cput[#1,#2][c]{%
  \vtop{\leftskip\z@\parindent\z@\raggedcenter\hm@d@#3}}}
\long\def\@@rflbox[#1,#2][r]#3{\@@rput[#1,#2][r]{%
  \vtop{\leftskip\z@\parindent\z@\raggedright\hm@d@#3}}}

\input{epsf.sty}

\def\epsfdirectory#1{\xdef\epsfdir@{#1}}

\def\epsfgetlitbb#1#2 #3 #4 #5]#6{\epsfgrab #2 #3 #4 #5 .\\%
   \epsfsetgraph{\epsfdir@#6}}%
\def\epsfnormal#1{\epsfgetbb{\epsfdir@#1}\epsfsetgraph{\epsfdir@#1}}%

\def\epsfsize#1#2{%
  \ifnum\epsfscale=1000
    \ifnum\epsfxsize=0
      \ifnum\epsfysize=0
      \else \computescale{\epsfysize}{#2}
      \fi
    \else \computescale{\epsfxsize}{#1}
    \fi
  \else
    \epsfxsize=#1
    \divide\epsfxsize by 1000
    \multiply\epsfxsize by \epsfscale
  \fi}

\epsfdirectory{}

\def\showfig#1#2{\epsfbox{#2}}
\def\fig@#1#2{\leavevmode{\framebox{\figstyl@\strut{ #1 }}}}

\def\figstyle#1{\def\figstyl@{#1}}

\figstyle{\family{cmr}\shape{n}\series{m}\selectfont\normalsize}

\def\showfigurestrue{\let\fig\showfig}
\def\showfiguresfalse{\let\fig\fig@}

\showfigurestrue

\makeatother 

  \let\epsilon\varepsilon
\let\textheta\theta      \let\theta\vartheta
          \let\phi\varphi
   \let\emptyset\varemptyset

\documentstyle[12pt]{article}

\makeatletter

\let\Larg@\Large
\let\hug@\huge

\def\usepackage#1{\input{#1.sty}}

\input{geom.sty}
\input{amstext.sty}

\def\rm{\ifmmode\else\protect\prm\fi}
\def\it{\ifmmode\else\protect\pit\fi}
\let\mathbf\bold

\def\r@adlabel#1#2{\global\@namedef{#1@\the\@key}{#2}}

\let\Large\Larg@
\let\huge\hug@

\def\smallskip{\vskip\smallskipamount}
\def\medskip{\vskip\medskipamount}
\def\bigskip{\vskip\bigskipamount}

\def\mytrivlist{\parsep\parskip\@nmbrlistfalse
  \my@trivlist \labelwidth\z@ \leftmargin\z@
  \itemindent\z@ \def\makelabel##1{##1}}

\def\my@trivlist{\global\@newlisttrue \@outerparskip\parskip}

\def\end#1{\csname end#1\endcsname\@checkend{#1}%
  \expandafter\endgroup\if@endpe\@doendpe\fi
  \if@ignore \global\@ignorefalse \ignorespaces\fi}

\def\maketitle{\par
 \begingroup
 \def\thefootnote{\fnsymbol{footnote}}
 \def\@makefnmark{\hbox 
 to 0pt{$^{\@thefnmark}$\hss}} 
 \if@twocolumn 
 \twocolumn[\@maketitle] 
 \else 
 \global\@topnum\z@ \@maketitle \fi\thispagestyle{plain}\@thanks
 \endgroup
 \setcounter{footnote}{0}
 \let\maketitle\relax
 \let\@maketitle\relax
 \gdef\@thanks{}\gdef\@author{}\gdef\@title{}\let\thanks\relax}

\def\@maketitle{ 
 \null
 \vskip 2em \begin{center}
 {\LARGE \@title \par} \vskip 1.5em {\large \lineskip .5em
\begin{tabular}[t]{c}\@author 
 \end{tabular}\par} 
 \vskip 1em {\large \@date} \end{center}
 \par
 \vskip 1.5em}
 
\def\partbeforeskip#1{\def\p@rtbeforeskip{#1}}
\def\partstyle#1{\def\p@rtstyl@{#1}}
\def\partdot#1{\def\partd@t{#1}}
\def\partafterskip#1{\def\p@rtafterskip{#1}}
\def\partintrostyle#1{\def\partintr@styl@{#1}}
\def\partintrodot#1{\def\partintr@dot{#1}}
\long\def\partintrosep#1{\long\def\partintr@sep{#1}}
\def\partnewpagetrue{\def\p@rtnewp@ge{\newpage}}
\def\partnewpagefalse{\long\def\p@rtnewp@ge{\par}}

\partbeforeskip{4ex}
\partstyle{\centering\Large\bf}
\partdot{}
\partafterskip{3ex}
\partintrostyle{\large}
\partintrodot{}
\partintrosep{\par}
\partnewpagefalse

\def\partname{Part}
\def\part{\p@rtnewp@ge\addvspace\p@rtbeforeskip\@afterindentfalse\secdef\@part\@spart}

\def\@part[#1]#2{\ifnum \c@secnumdepth >-1\relax  
        \refstepcounter{part}                     
        \def\@tempa{\addcontentsline{toc}{part}}  %
        \expandafter\@tempa\expandafter{\thepart  
          \hspace{1em}#1}\else                    
        \addcontentsline{toc}{part}{#1}\fi        
   {\p@rtstyl@                       
    \ifnum \c@secnumdepth >-1\relax        
      {\partintr@styl@\partname\ \thepart  
       \partintr@dot}\partintr@sep\nobreak 
    \fi                                    
    #2\partd@t\markboth{}{}\par}
    \nobreak                       
    \vskip\p@rtafterskip           
   \@afterheading                  
    }                              

\def\@spart#1{{\p@rtcentering\p@rtstyl@                      
    #1\partd@t\par}                 
    \nobreak                        
    \vskip\p@rtafterskip            
    \@afterheading                  
  }                                 

\newif\ifsection@ftind
\newif\ifsection@ftpar

\def\sectionbeforeskip#1{\def\s@ctbeforeskip{#1}}
\def\sectionstyle#1{\def\s@ctstyl@{#1}}
\def\sectiondot#1{\def\sectiond@t{#1}}
\def\sectionafterskip#1{\def\s@ctafterskip{#1}}
\def\sectionintrostyle#1{\def\sectionintr@styl@{#1}}
\def\sectionintro#1{\def\sectionintr@{#1}}
\def\sectionintrodot#1{\def\sectionintr@dot{#1}}
\def\sectionintrosep#1{\def\sectionintr@sep{#1}}
\def\sectionindenttrue{\def\s@ctind{\parindent}}
\def\sectionindentfalse{\def\s@ctind{\z@}}
\def\sectionafterindenttrue{\section@ftindtrue}
\def\sectionafterindentfalse{\section@ftindfalse}
\def\sectionafternewlinetrue{\section@ftpartrue}
\def\sectionafternewlinefalse{\section@ftparfalse}

\newif\ifsubsection@ftind
\newif\ifsubsection@ftpar

\def\subsectionbeforeskip#1{\def\ss@ctbeforeskip{#1}}
\def\subsectionstyle#1{\def\ss@ctstyl@{#1}}
\def\subsectiondot#1{\def\subsectiond@t{#1}}
\def\subsectionafterskip#1{\def\ss@ctafterskip{#1}}
\def\subsectionintrostyle#1{\def\subsectionintr@styl@{#1}}
\def\subsectionintro#1{\def\subsectionintr@{#1}}
\def\subsectionintrodot#1{\def\subsectionintr@dot{#1}}
\def\subsectionintrosep#1{\def\subsectionintr@sep{#1}}
\def\subsectionindenttrue{\def\ss@ctind{\parindent}}
\def\subsectionindentfalse{\def\ss@ctind{\z@}}
\def\subsectionafterindenttrue{\subsection@ftindtrue}
\def\subsectionafterindentfalse{\subsection@ftindfalse}
\def\subsectionafternewlinetrue{\subsection@ftpartrue}
\def\subsectionafternewlinefalse{\subsection@ftparfalse}

\newif\ifsubsubsection@ftind
\newif\ifsubsubsection@ftpar

\def\subsubsectionbeforeskip#1{\def\sss@ctbeforeskip{#1}}
\def\subsubsectionstyle#1{\def\sss@ctstyl@{#1}}
\def\subsubsectiondot#1{\def\subsubsectiond@t{#1}}
\def\subsubsectionafterskip#1{\def\sss@ctafterskip{#1}}
\def\subsubsectionintrostyle#1{\def\subsubsectionintr@styl@{#1}}
\def\subsubsectionintro#1{\def\subsubsectionintr@{#1}}
\def\subsubsectionintrodot#1{\def\subsubsectionintr@dot{#1}}
\def\subsubsectionintrosep#1{\def\subsubsectionintr@sep{#1}}
\def\subsubsectionindenttrue{\def\sss@ctind{\parindent}}
\def\subsubsectionindentfalse{\def\sss@ctind{\z@}}
\def\subsubsectionafterindenttrue{\subsubsection@ftindtrue}
\def\subsubsectionafterindentfalse{\subsubsection@ftindfalse}
\def\subsubsectionafternewlinetrue{\subsubsection@ftpartrue}
\def\subsubsectionafternewlinefalse{\subsubsection@ftparfalse}

\newif\ifparagraph@ftind
\newif\ifparagraph@ftpar

\def\paragraphbeforeskip#1{\def\p@rbeforeskip{#1}}
\def\paragraphstyle#1{\def\p@rstyl@{#1}}
\def\paragraphdot#1{\def\paragraphd@t{#1}}
\def\paragraphafterskip#1{\def\p@rafterskip{#1}}
\def\paragraphintrostyle#1{\def\paragraphintr@styl@{#1}}
\def\paragraphintro#1{\def\paragraphintr@{#1}}
\def\paragraphintrodot#1{\def\paragraphintr@dot{#1}}
\def\paragraphintrosep#1{\def\paragraphintr@sep{#1}}
\def\paragraphindenttrue{\def\p@rind{\parindent}}
\def\paragraphindentfalse{\def\p@rind{\z@}}
\def\paragraphafterindenttrue{\paragraph@ftindtrue}
\def\paragraphafterindentfalse{\paragraph@ftindfalse}
\def\paragraphafternewlinetrue{\paragraph@ftpartrue}
\def\paragraphafternewlinefalse{\paragraph@ftparfalse}

\newif\ifsubparagraph@ftind
\newif\ifsubparagraph@ftpar

\def\subparagraphbeforeskip#1{\def\sp@rbeforeskip{#1}}
\def\subparagraphstyle#1{\def\sp@rstyl@{#1}}
\def\subparagraphdot#1{\def\subparagraphd@t{#1}}
\def\subparagraphafterskip#1{\def\sp@rafterskip{#1}}
\def\subparagraphintrostyle#1{\def\subparagraphintr@styl@{#1}}
\def\subparagraphintro#1{\def\subparagraphintr@{#1}}
\def\subparagraphintrodot#1{\def\subparagraphintr@dot{#1}}
\def\subparagraphintrosep#1{\def\subparagraphintr@sep{#1}}
\def\subparagraphindenttrue{\def\sp@rind{\parindent}}
\def\subparagraphindentfalse{\def\sp@rind{\z@}}
\def\subparagraphafterindenttrue{\subparagraph@ftindtrue}
\def\subparagraphafterindentfalse{\subparagraph@ftindfalse}
\def\subparagraphafternewlinetrue{\subparagraph@ftpartrue}
\def\subparagraphafternewlinefalse{\subparagraph@ftparfalse}

\sectionbeforeskip{\bigskipamount}
\sectionstyle{\large\bf}
\sectiondot{}
\sectionafterskip{.5\bigskipamount}
\sectionintrostyle{}
\sectionintro{}
\sectionintrodot{.}
\sectionintrosep{1.25ex}
\sectionindentfalse
\sectionafterindenttrue
\sectionafternewlinetrue

\subsectionbeforeskip{.8\bigskipamount}
\subsectionstyle{\normalsize\bf}
\subsectiondot{}
\subsectionafterskip{.4\bigskipamount}
\subsectionintrostyle{}
\subsectionintro{}
\subsectionintrodot{.}
\subsectionintrosep{1.25ex}
\subsectionindentfalse
\subsectionafterindenttrue
\subsectionafternewlinetrue

\subsubsectionbeforeskip{.6\bigskipamount}
\subsubsectionstyle{\normalsize\bf}
\subsubsectiondot{}
\subsubsectionafterskip{.3\bigskipamount}
\subsubsectionintrostyle{}
\subsubsectionintro{}
\subsubsectionintrodot{.}
\subsubsectionintrosep{1.25ex}
\subsubsectionindentfalse
\subsubsectionafterindenttrue
\subsubsectionafternewlinetrue

\paragraphbeforeskip{.5\bigskipamount}
\paragraphstyle{\normalsize\bf}
\paragraphdot{.}
\paragraphafterskip{1.25ex}
\paragraphintrostyle{}
\paragraphintro{}
\paragraphintrodot{.}
\paragraphintrosep{1.25ex}
\paragraphindentfalse
\paragraphafterindenttrue
\paragraphafternewlinefalse

\subparagraphbeforeskip{.5\bigskipamount}
\subparagraphstyle{\normalsize\bf}
\subparagraphdot{.}
\subparagraphafterskip{1.25ex}
\subparagraphintrostyle{}
\subparagraphintro{}
\subparagraphintrodot{.}
\subparagraphintrosep{1.25ex}
\subparagraphindenttrue
\subparagraphafterindenttrue
\subparagraphafternewlinefalse

\let\@partoken\par
\long\def\@@gobble#1{}
\def\ignorepar{\@ifnextchar\@partoken{\expandafter\ignorepar\@@gobble}{\ignorespaces}}

\def\@startsection#1#2#3#4#5#6{
   \@tempskipa #4\relax
   \csname if#1@ftind\endcsname\@afterindenttrue\else\@afterindentfalse\fi
   \advance\@tempskipa by\presection
   \if@nobreak \everypar{}\else
     \addpenalty{\@secpenalty}\addvspace{\@tempskipa}%
     \allowbreak\vskip -\presection \fi \@ifstar
     {\@ssect{#1}{#2}{#3}{#4}{#5}{#6}}{\@dblarg{\@sect{#1}{#2}{#3}{#4}{#5}{#6}}}}

\def\@sect#1#2#3#4#5#6[#7]#8{\def\object@type{#1}%
   \ifnum #2>\c@secnumdepth\def\@svsec{}\def\@tempb{}%
      \else\refstepcounter{#1}\def\@svsec{{\csname #1intr@styl@\endcsname%
        {\csname #1intr@\endcsname}\csname the#1\endcsname%
        \csname #1intr@dot\endcsname\kern\csname #1intr@sep\endcsname}}%
        \edef\@tempb{\noexpand\numberline{\csname the#1\endcsname}}\fi%
   \def\@tempa{\addcontentsline{toc}{#1}}%
   \csname if#1@ftpar\endcsname%
      \begingroup #6\relax%
        \@hangfrom{\hskip #3\relax\@svsec}{\interlinepenalty \@M{#8}%
        \csname #1d@t\endcsname\par}%
      \endgroup%
      \csname #1mark\endcsname{#7}%
      \expandafter\@tempa\expandafter{\@tempb #7}%
      \ifautolabel\label*{#8}\fi%
   \else%
      \def\@svsechd{#6\hskip #3\relax%
         \@svsec{#8}%
         \csname #1d@t\endcsname%
         \csname #1mark\endcsname{#7}%
         \expandafter\@tempa\expandafter{\@tempb #7}%
         \ifautolabel\label*{#8}\fi}\fi%
   \@xsect{#1}{#5}\ignorepar}

\def\@ssect#1#2#3#4#5#6#7{%
   \ifnum #2>\c@secnumdepth\def\@tempb{}\else \def\@tempb{\numberline{}}\fi%
     \def\@tempa{\addcontentsline{toc}{s#1}}%
     \csname if#1@ftpar\endcsname
        \begingroup #6\relax
           \@hangfrom{\hskip #3}{\interlinepenalty \@M{#7}%
           \csname #1d@t\endcsname\par}%
        \endgroup
        \csname s#1mark\endcsname{#7}%
        \ifstarredcontents\expandafter\@tempa\expandafter{\@tempb #7}\fi%
        \ifautolabel\label*{#7}\fi%
     \else%
        \def\@svsechd{#6\hskip #3\relax{#7}%
        \csname #1d@t\endcsname%
        \csname s#1mark\endcsname{#7}%
        \ifautolabel\label*{#7}\fi}\fi
   \@xsect{#1}{#5}\ignorepar}

\def\@xsect#1#2{
   \csname if#1@ftpar\endcsname 
       \par \nobreak \vskip #2\relax \@afterheading
    \else \global\@nobreakfalse \global\@noskipsectrue
       \everypar{\if@noskipsec \global\@noskipsecfalse
                   \clubpenalty\@M \hskip -\parindent
                   \begingroup \@svsechd \endgroup \unskip
                   \hskip #2\relax  
                  \else \clubpenalty \@clubpenalty
                    \everypar{}\fi}\fi\ignorespaces}

\def\section{\@startsection{section}{1}{\s@ctind}
  {\s@ctbeforeskip}{\s@ctafterskip}{\s@ctstyl@}}
\def\subsection{\@startsection{subsection}{2}{\ss@ctind}
  {\ss@ctbeforeskip}{\ss@ctafterskip}{\ss@ctstyl@}}
\def\subsubsection{\@startsection{subsubsection}{3}{\sss@ctind}
  {\sss@ctbeforeskip}{\sss@ctafterskip}{\sss@ctstyl@}}
\def\paragraph{\@startsection{paragraph}{4}{\p@rind}
  {\p@rbeforeskip}{\p@rafterskip}{\p@rstyl@}}
\def\subparagraph{\@startsection{subparagraph}{4}{\sp@rind}
  {\sp@rbeforeskip}{\sp@rafterskip}{\sp@rstyl@}}

\def\statementabove#1{\def\th@bove{#1}}
\def\statementstyle#1{\def\thstyl@{#1}}
\def\statementbelow#1{\def\thb@low{#1}}
\def\statementindentfalse{\let\thind@nt\relax}
\def\statementindenttrue{\let\thind@nt\indent}

\def\statementintrostyle#1{\def\thintr@style{#1}}
\def\statementintrodot#1{\def\thintr@dot{#1}}
\def\statementintrosep#1{\def\thintr@sep{#1}}
\def\statementintrobrackets#1#2{\def\thintr@left{#1}\def\thintr@right{#2}}

\statementabove{\medskip}
\statementstyle{\sl}
\statementbelow{\medskip}
\statementindenttrue

\statementintrostyle{\normalshape\bf}
\statementintrodot{.}
\statementintrosep{\kern1.25ex}
\statementintrobrackets{(}{)}

\def\@thskip{\dimen100\lastskip\vskip-\dimen100%
  \th@bove\dimen101\lastskip\vskip-\dimen101%
  \ifdim\dimen100>\dimen101\else\dimen100\dimen101\fi\vskip\dimen100\vskip0pt}

\long\def\@@newtheorem#1#2#3{%
  \newenvironment{#3}%
    {\def\object@type{#3}\par\@thskip%
     \@ifnextchar[{\@enva{#3}{\thstyl@#1{#2}}}{\@envb{#3}{\thstyl@#1{#2}}}}%
    {\end{#3@}}%
  \@ifnextchar[{\@nothm{#3}}{\@nnthm{#3}}}

\def\@nothm#1[#2]#3{%
  \@ifundefined{c@#2}{\@latexerr{No theorem environment `#2' defined}\@eha}%
  {\expandafter\@ifdefinable\csname #1@\endcsname
  {\global\@namedef{the#1}{\@nameuse{the#2}}%
   \global\@namedef{c@#1}{\@nameuse{c@#2}}
   \global\@namedef{p@#1}{\@nameuse{p@#2}}
   \global\@namedef{#1@}{\@nnnthm{#2}{#3}}%
   \global\@namedef{end#1@}{\@endtheorem}}}}

\def\@nnnthm#1#2{\refstepcounter
    {#1}\@ifnextchar[{\@ynnnthm{#1}{#2}}{\@xnnnthm{#1}{#2}}}
 
\def\@xnnnthm#1#2{\@begintheorem{#2}{\csname the#1\endcsname}\ignorespaces}
\def\@ynnnthm#1#2[#3]{\@opargbegintheorem{#2}{\csname the#1\endcsname}{#3}\ignorespaces}

\def\renewtheorem{\@ifnextchar[{\@renewtheorem}{\@renewtheorem[{}{}]}}

\long\def\@renewtheorem[#1]{\@@renewtheorem#1}

\long\def\@@renewtheorem#1#2#3{%
  \expandafter\let\csname#3@\endcsname\undefined
  \renewenvironment{#3}%
    {\def\object@type{#3}\par\@thskip%
     \@ifnextchar[{\@enva{#3}{\thstyl@#1{#2}}}{\@envb{#3}{\thstyl@#1{#2}}}}%
    {\end{#3@}}%
  \@ifnextchar[{\@nothm{#3}}{\@nnthm{#3}}}

\def\@begintheorem#1#2{\@opargbegintheorem{#1}{#2}{}}

\def\@opargbegintheorem#1#2#3{%
        \edef\@tempx{#1}%
        \expandafter\let\expandafter\@tempy#2
        \def\@tempz{#3}%
        \mytrivlist\item[\thind@nt\hskip\labelsep%
        {\thintr@style%
          #1\ifx\@tempx\@empty\else\ifx\@tempy\relax\else\kern1ex\fi\fi#2%
          \ifx\@tempz\@empty%
            \ifx\@tempx\@empty\ifx\@tempy\relax%
            \else\thintr@dot\thintr@sep\fi\else\thintr@dot\thintr@sep\fi%
            \else%
            \ifx\@tempx\@empty\ifx\@tempy\relax\else\kern1ex\fi\else\kern1ex\fi%
           \thintr@left{#3}\thintr@right\thintr@dot\thintr@sep\fi}%
            \hskip-\labelsep]%
        \ifautolabel\label*{#3}\fi}

\def\@endtheorem{\strut\endtrivlist\thb@low}

\def\proofname{Proof}

\def\proofabove#1{\def\pf@bove{#1}}
\def\proofstyle#1{\def\pfstyl@{#1}}
\def\proofbelow#1{\def\pfb@low{#1}}
\def\proofindentfalse{\let\pfind@nt\relax}
\def\proofindenttrue{\let\pfind@nt\indent}

\def\proofintrostyle#1{\def\pfintr@style{#1}}
\def\proofintrodot#1{\def\pfintr@dot{#1}}
\def\proofintrosep#1{\def\pfintr@sep{#1}}
\def\proofintrobrackets#1#2{\def\pfintr@left{#1}\def\pfintr@right{#2}}

\proofabove{\medskip}
\proofstyle{}
\proofbelow{\medskip}
\proofindenttrue

\proofintrostyle{\sl}
\proofintrodot{.}
\proofintrosep{\kern1.25ex}
\proofintrobrackets{of\kern1ex}{}

\def\@pfskip{\dimen100\lastskip\vskip-\dimen100%
  \pf@bove\dimen101\lastskip\vskip-\dimen101%
  \ifdim\dimen100>\dimen101\else\dimen100\dimen101\fi\vskip\dimen100\vskip0pt}

\renewenvironment{proof}%
  {\@pfskip\mytrivlist\item[\pfind@nt]\@ifnextchar[{\pro@f}{\pro@f[\prooftag]}}
  {\ifvoid\provedbox\else\hproved\fi\endtrivlist\pfb@low}

\def\pro@f[#1]{\setbox\provedbox\hbox{\provedboxcontents{#1}}\proofintro{#1}}

\def\proofintro#1{\expandafter\def\expandafter\@tempa\expandafter{#1}%
  {\pfintr@style{\proofname\ifx\@tempa\empty\else\kern1ex\pfintr@left{#1}%
  \pfintr@right\fi}\pfintr@dot\pfintr@sep}\pfstyl@\ignorespaces}

\def\provedmark#1{\def\prm@rk{#1}}
\def\provedsep#1{\def\prs@p{#1}}

\provedmark{$\square$}
\provedsep{\kern1.25ex}

\def\provedtexttrue{\def\prb@x##1{\fbox{\small##1}}}
\def\provedtextfalse{\def\prb@x##1{\prm@rk}}
\def\provedmarkrighttrue{\let\prhf@l\hfill}
\def\provedmarkrightfalse{\let\prhf@l\relax}

\provedtextfalse
\provedmarkrightfalse

\def\provedboxcontents#1{\expandafter\def\expandafter\@tempa\expandafter{#1}%
  \ifx\@tempa\empty\prm@rk\else\prb@x{#1}\fi}

\def\proved{\ifmmode\eqno{\box\provedbox}\else\hproved\fi}

\def\hproved{\unskip\nobreak\prhf@l\penalty50\prs@p\hbox{}\nobreak\prhf@l
  \box\provedbox{\finalhyphendemerits=0\par}}

\def\captionstyle#1{\def\c@ptstyl@{#1}}
\def\captionintrostyle#1{\def\c@pintr@style{#1}}
\def\captionintrodot#1{\def\c@pintr@dot{#1}}
\def\captionintrosep#1{\def\c@pintr@sep{#1}}

\captionstyle{\small\sf}
\captionintrostyle{\bf}
\captionintrodot{.}
\captionintrosep{\hskip1.25ex}

\long\def\@makecaption#1#2{%
    \vskip\captionskip
    \setbox\@tempboxa\hbox{%
      \ifproofing\@ifundefined{the@label}{}
        {\hbox to 0pt{\vbox to 0pt{\vss\hbox{\tiny\the@label}\bigskip}\hss}}\fi
      \c@ptstyl@{\c@pintr@style #1\c@pintr@dot}\ignorespaces #2}%
    \@captionwidth=\hsize \advance\@captionwidth-2\@captionmargin
    \ifdim \wd\@tempboxa >\@captionwidth {%
        \rightskip=\@captionmargin\leftskip=\@captionmargin
        \unhbox\@tempboxa\par}%
      \else
        \hbox to\hsize{\hfil\box\@tempboxa\hfil}%
    \fi}

\def\end@Float#1{%
  \expandafter\caption\expandafter[\the@title]{%
   {\c@pintr@style%
   \ifx\the@caption\empty\ifx\the@title\empty
   \else\c@pintr@sep\fi\else\c@pintr@sep\fi
    \the@title\ifx\the@caption\empty%
     \expandafter\label\expandafter*\expandafter{\the@label}%
    \else\ifx\the@title\empty%
     \expandafter\label\expandafter*\expandafter{\the@label}%
    \else\c@pintr@dot\c@pintr@sep%
     \expandafter\label\expandafter*\expandafter{\the@label}\fi\fi}%
   \ignorespaces\the@caption}%
  \end{#1}}

\renewenvironment{Figure}{\@ifnextchar[%
  {\@myFloat{figure}}{\@myFloat{figure}[htbp]}}{\end@Float{figure}}

\def\@myFloat#1[#2]#3{%
  \begin{#1}[#2]\def\the@label{#3}}

\def\showfig{\showfigurestrue\fig}
\def\fig#1{\@ifnextchar[{\@fig{#1}}{\@fig{#1}[0pt]}}

\def\@fig#1[#2]#3{\@ifnextchar[{\@@fig{#1}[#2]{#3}}{\@@fig{#1}[#2]{#3}[0pt]}}
\def\@@fig#1[#2]#3[#4]#5#6{%
  \def\the@title{#5}\def\the@caption{#6}\centerline{\fig@{#1}{#2}{#3}}\vskip#4}

\def\fig@@#1#2#3{\leavevmode{\figstyl@\vrule width 0pt height 1.8ex%
 \smash{\framebox{\strut\def\@temp{#1}\ifx\@temp\@empty{ #3 }\else{ #1 }\fi}}}}
\def\fig@@@#1#2#3{\leavevmode\kern#2\epsfbox{#3}}

\def\figstyle#1{\def\figstyl@{#1}}

\figstyle{\family{cmr}\shape{n}\series{m}\selectfont\normalsize}

\def\oldFigure{%
  \renewenvironment{Figure}{\@ifnextchar[%
    {\@Float{figure}}{\@Float{figure}[htbp]}}{\end@Float{figure}}
    \let\showfig\@ldshowfig \let\fig\@ldfig
    \let\showfigurestrue\@ldshowfigurestrue
    \let\showfiguresfalse\@ldshowfiguresfalse
    \showfiguresfalse}

\def\@ldshowfig#1#2{\epsfbox{#2}}
\def\@ldfig@#1#2{\leavevmode{\framebox{\figstyl@\strut{ #1 }}}}
\def\@ldshowfigurestrue{\let\fig\@ldshowfig}
\def\@ldshowfiguresfalse{\let\fig\@ldfig@}

\newcounter{diagram}

\let\thediagram\theequation
\def\ftype@diagram{2}
\def\ext@diagram{lod}
\def\diagram{\@float{diagram}}
\let\enddiagram\end@float
\newif\if@diagnum

\def\diag#1{\@ifnextchar[{\@diag{#1}}{\@diag{#1}[0pt]}}

\def\@diag#1[#2]#3{\@ifnextchar[{\@@diag{#1}[#2]{#3}}{\@@diag{#1}[#2]{#3}[0pt]}}

\def\@@diag#1[#2]#3[#4]#5{
  \def\the@tag{#5}\@eqnswtrue%
  \centerline{\setbox0\hbox{\diag@{#1}{#2}{#3}}
  \dimen0 -0.5\wd0\dimen1 0.5\ht0\box0%
  \advance\dimen0 0.5\hsize\advance\dimen0 -\rightskip\advance\dimen1 #4%
  \let\@currentlabel\the@tag%
  \setbox0\hbox to 0pt{\hss\family{cmr}\shape{n}\series{m}\selectfont(\the@tag)}%
  \ifx\the@tag\@empty\refstepcounter{equation}\let\@currentlabel\theequation%
    \setbox0\hbox to 0pt{\hss\family{cmr}\shape{n}\series{m}\selectfont(\thediagram)}\fi%
  \if@eqnsw\else\let\@currentlabel\relax\setbox0\hbox to 0pt{}\fi%
  \advance\dimen1 -0.5\ht0%
  \put[\dimen0,\dimen1][l]{%
    \box0\expandafter\label\expandafter*\expandafter{\the@label}\kern0.15em}}}

\def\diag@@#1#2#3{\leavevmode{\diagstyl@\vrule width 0pt height 1.8ex%
 \smash{\framebox{\strut\def\@temp{#1}\ifx\@temp\@empty{ #3 }\else{ #1 }\fi}}}}
\def\diag@@@#1#2#3{\leavevmode\kern#2\epsfbox{#3}}

\def\diagstyle#1{\def\diagstyl@{#1}}

\diagstyle{\family{cmr}\shape{n}\series{m}\selectfont\normalsize}

\def\showfiguresfalse{\let\fig@\fig@@}
\def\showfigurestrue{\let\fig@\fig@@@}

\def\showdiagramsfalse{\let\diag@\diag@@}
\def\showdiagramstrue{\let\diag@\diag@@@}

\def\n@number{\@eqnswfalse\let\@currentlabel\relax\let\the@tag\relax}

\def\equation{$$
  \@eqnswtrue\def\object@type{equation}\let\nonumber\n@number%
  \advance\c@equation1\edef\@currentlabel{\theequation}\advance\c@equation-1%
  \def\the@tag{\refstepcounter{equation}\eqno\hbox{\@eqnnum}}}

\def\tag#1{\edef\@currentlabel{#1}\def\the@tag{\eqno\hbox{\reset@font\rm(#1)}}}

\def\endequation{\the@tag$$
  \global\@ignoretrue}

\let\it@m\item
\def\item{\@ifnextchar[{\item@}{\item@@}}
\def\item@[#1]{\it@m[#1]\vskip-\lastskip\vskip\itemsep}
\def\item@@{\it@m\vskip-\lastskip\vskip\itemsep}
\def\s@titemsep{\@ifnextchar[{\s@@titemsep}{\relax}}
\def\s@@titemsep[#1]{\itemsep#1}

\let\@itemize\itemize
\let\@enditemize\enditemize
\renewenvironment{itemize}
{\@itemize\itemsep3pt\parsep0pt\topsep0pt\partopsep0pt\s@titemsep}
{\@enditemize\vskip-\lastskip\vskip\itemsep}

\let\@enumerate\enumerate
\let\@endenumerate\endenumerate
\renewenvironment{enumerate}
{\@enumerate\itemsep3pt\parsep0pt\topsep0pt\partopsep0pt\s@titemsep}
{\@endenumerate\vskip-\lastskip\vskip\itemsep}

\let\@description\description
\let\@enddescription\enddescription
\renewenvironment{description}
{\@description\itemsep3pt\parsep0pt\topsep0pt\partopsep0pt\s@titemsep}
{\@enddescription\vskip-\lastskip\vskip\itemsep}

\topsep\dimen200
\partopsep\dimen201

\@definecounter{bibenumi}

\def\thebibliography#1{%
 \section*{\refname}\vskip-\lastskip%
 \list{[\arabic{bibenumi}]}{\topsep0pt\settowidth\labelwidth{[#1]}%
 \leftmargin\labelwidth\advance\leftmargin\labelsep\usecounter{bibenumi}}%
 \def\newblock{\hskip .11em plus .33em minus .07em}%
 \sloppy\clubpenalty4000\widowpenalty4000\sfcode`\.=1000\relax}

\makeatother

\makeatletter
\textfloatsep\floatsep
\makeatother

\abovedisplayskip\smallskipamount
\belowdisplayskip\smallskipamount
\abovedisplayshortskip\smallskipamount
\belowdisplayshortskip\smallskipamount

\proofingfalse
\autolabelfalse

\showfigurestrue
\showdiagramstrue

\newtheorem{stat}{\statname}  \unnumbered{stat}

\newtheorem{nstat}{\nstatname}[section]

\newtheorem{definition}[nstat]{Definition}
\newtheorem{lemma}[nstat]{Lemma}
\newtheorem{proposition}[nstat]{Proposition}
\newtheorem{theorem}[nstat]{Theorem}

\newtheorem{remark}[nstat]{Remark}

\let\ns\normalshape

\papersize{216truemm}{279truemm}
\margins{33truemm}{20.5truemm}
\advance\voffset-5mm
\advance\hoffset7.5truemm
\advance\footskip2.5truemm

\headline{\hfill}
\footline{\small\hfill--\kern1ex\thepage\kern1ex--\hfill}

\textfloatsep\floatsep

\abovedisplayskip\smallskipamount
\belowdisplayskip\smallskipamount
\abovedisplayshortskip\smallskipamount
\belowdisplayshortskip\smallskipamount

\showfigurestrue

\sectionbeforeskip{1.5\bigskipamount}
\sectionstyle{\centering\normalsize\sc}
\sectionafterskip{\bigskipamount}
\sectionafterindenttrue

\subsectionbeforeskip{\bigskipamount}
\subsectionstyle{\normalsize\bf}
\subsectionafterskip{.5\bigskipamount}
\subsectionafterindenttrue

\paragraphbeforeskip{\medskipamount}
\paragraphstyle{\bf\sl}
\paragraphafterskip{1.25ex}
\paragraphindenttrue
\paragraphdot{.}

\statementintrostyle{\bf}

\renewtheorem{lemma}[gnstat]{Lemma}
\renewtheorem{proposition}[gnstat]{Proposition}
\renewtheorem{theorem}[gnstat]{Theorem}
\renewtheorem{corollary}[gnstat]{Corollary}
\renewtheorem{problem}[gnstat]{Problem}

\renewtheorem[{\ns}{}]{definition}[gnstat]{Definition}
\renewtheorem[{\ns}{}]{remark}[gnstat]{Remark}

\captionstyle{\small\ns}
\captionintrostyle{\sc}

\def~{\kern0.33em}

\let\mycal\cal
\def\cal#1{{\mycal #1}}

\let\mymathrm\mathrm
\def\mathrm#1{{\mymathrm #1}}

\let\texbf\bf
\def\bf{\texbf\boldmath}

\def\leftwavearrow{\mathrel{\raise0.8pt\hbox{\epsfbox{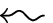}}}}
\def\rightwavearrow{\mathrel{\raise0.8pt\hbox{\epsfbox{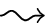}}}}
\def\leftrightwavearrow{\mathrel{\raise0.8pt\hbox{\epsfbox{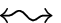}}}}
\def\longleftwavearrow{\mathrel{\raise0.8pt\hbox{\epsfbox{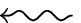}}}}
\def\longrightwavearrow{\mathrel{\raise0.8pt\hbox{\epsfbox{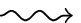}}}}
\def\longleftrightwavearrow{\mathrel{\raise0.8pt\hbox{\epsfbox{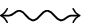}}}}

\let\rightmove\rightwavearrow

\def\varemptyset{{\text{\raise.21ex\hbox{$\not$}}\mkern.15mu\mathrm{O}\mkern.15mu}}
\let\emptyset\varemptyset
\let\0\emptyset

\let\theta\textheta

\let\bar\widebar
\let\hat\widehat
\let\tilde\widetilde

\newcommand{\A}{\cal A}

\newcommand{\E}{\cal E}
\newcommand{\G}{\cal G}

\newcommand{\M}{\cal M}
\newcommand{\R}{\cal R}
\renewcommand{\S}{\cal S}
\newcommand{\T}{\cal T}

\newcommand{\Cl}{\mathop{\mathrm{Cl}}\nolimits}
\newcommand{\Int}{\mathop{\mathrm{Int}}\nolimits}
\newcommand{\Stab}{\mathop{\mathrm{Stab}}\nolimits}
\newcommand{\Orb}{\mathop{\mathrm{Orb}}\nolimits}

\newcommand{\len}{\mathop{\mathrm{len}}\nolimits}
\newcommand{\ord}{\mathop{\mathrm{ord}}\nolimits}

\newcommand{\LCM}{\mathop{\mathrm{LCM}}\nolimits}

\begin{document}

\title{\large\bf
AUTOMORPHISMS OF TRIVALENT GRAPHS
\label{Version 1.0 / \today}}
\author{\normalsize\sc Silvia Benvenuti\\
\normalsize\sl Scuola di Scienze e Tecnologie\\
\normalsize\sl Universit\`a di Camerino -- Italia\\
\small\tt silvia.benvenuti@unicam.it 
\and
\normalsize\sc Riccardo Piergallini\\
\normalsize\sl Scuola di Scienze e Tecnologie\\
\normalsize\sl Universit\`a di Camerino -- Italia\\
\small\tt riccardo.piergallini@unicam.it}

\date{}

\vglue1.2cm

\maketitle

\bigskip

\begin{abstract}
\baselineskip13.5pt
\medskip

\noindent
Let $\G_{g,b}$ be the set of all uni/trivalent graphs representing the combinatorial
structures of pant decompositions of the oriented surface $\Sigma_{g,b}$ of genus $g$ with 
$b$ boundary components. We describe the set $\A_{g,b}$ of all automorphisms of graphs in
$\G_{g,b}$ showing that, up to suitable moves changing the graph within $\G_{g,b}$, any
such automorphism can be reduced to elementary switches of adjacent edges.



\end{abstract}

\maketitle

\bigskip

\section*{Introduction}

Let $\G_{g,b}$ be the set of all connected non-oriented uni/trivalent graphs $\Gamma$,
with first Betti number $\beta_1(\Gamma) = g \geq 0$, at least one trivalent vertex and $b
\geq 0$ univalent vertices. It is immediate to see that such a graph $\Gamma$ exists if
and only if $(g,b) \neq (0,0),(0,1),(0,2),(1,0)$, and in this case $\Gamma$ has $2g-2+b$
trivalent vertices and $3g-3+2b$ edges. In the following, the pair $(g,b)$ will be always
assumed to satisfy that restriction.

Given any $\Gamma \in \G_{g,b}$, we denote by $\A(\Gamma)$ the group of all the
automorphisms of $\Gamma$ as a non-oriented graph. In this paper we are interested in the
structure of the set $\A_{g,b} = \cup_{\Gamma \in \G_{g,b}}\A(\Gamma)$ of all the
automorphisms of graphs in $\G_{g,b}$. More precisely, we will prove that, up to certain
$F$-moves changing the graph within $\G_{g,b}$, any such automorphism can be reduced to
elementary switches of adjacent edges.

The drive for this paper comes from a problem we met while working on pant decompositions
of surfaces in \cite{BP08}. Namely, in that paper we build up an infinite simply connected
2-dimensional complex $\R_{g,b}$, codifying all the pant decompositions on the connected
compact oriented surface $\Sigma_{g,b}$ of genus $g$ with $b$ boundary components, as well
as all the moves relating them and all the relations between those moves. The construction
is subdivided into two independent steps.

First a finite simply connected 2-dimensional complex $\S_{g,b}$ is described, which
codifies the {\sl combinatorial structures} of pant decompositions of $\Sigma_{g,b}$. The
combinatorial structure of a pant decomposition $D$ consists of the information concerning
only the incidence relations between pants, and it is encoded by the dual graph $\Gamma_D
\in \G_{g,b}$, with the trivalent vertices corresponding to the pants, the univalent
vertices corresponding to the boundary components, and the edges corresponding to the
cutting curves. Afterwards, by using a presentation of the mapping class group $\M_{g,b}$
of $\Sigma_{g,b}$, we get the desired complex $\R_{g,b}$.
 
Due to technical issues, the switch from $\S_{g,b}$ to $\R_{g,b}$ is not a ``direct''
one, but it requires instead a cumbersome intermediate step. This involves in the
construction of two additional complexes $\tilde{\S}_{g,b}$ and $\tilde{\R}_{g,b}$,
codifying respectively the {\sl decorated combinatorial structures} and the {\sl decorated
pant decompositions}. Here, a decoration of a pant decomposition $D$ is a numbering of its
cutting curves, and a decoration of the corresponding combinatorial structure is a
numbering of the edges of $\Gamma_D$.

This method allows us to sidestep a crucial problem in the direct transition from
$\S_{g,b}$ to $\R_{g,b}$, that is the study of the stabilization subgroup $\Stab(D) \in
\M_{g,b}$ of any given pant decomposition $D$. In fact, the proof of the simple
connectedness of $\R_{g,b}$ given in \cite{BP08} would become less elaborate based on the
fact that the fiber $p^{-1}(\Gamma_D)$ of the natural projection $p: \R_{g,b} \to
\S_{g,b}$ is trivial at $\R_{g,b}$ on the level of the fundamental group. In other words,
we need to show that any loop based at $D$ and contained in $p^{-1}(\Gamma_D)$ is
contractible in $\R_{g,b}$. Now, such a loop corresponds to an element of $\Stab(D)$, that
is a symmetry of $D$, which induces a (possibly trivial) combinatorial symmetry of
$\Gamma_D$. Actually, once the combinatorial symmetries of $\Gamma_D$ are known, one can
reconstruct those of $D$ in a straightforward way, by adding some easy topological
information.

At this point, it should be clear the motivation for the study carried out in the present
paper of the structure of $\A_{g,b}$, and in particular of the inclusion $\A(\Gamma)
\subset \A_{g,b}$ of the group of combinatorial symmetries $\A(\Gamma)$ of a given $\Gamma
\in \G_{g,b}$. In fact, the result we obtain here will enable us to give a much simpler
contruction of $\R_{g,b}$ than in \cite{BP08}, and a more direct proof of its simple
connectedness starting from $\S_{g,b}$ without involving decorations.

However, the study of the automorphisms of uni/trivalent graphs has its own interest
independently from that specific application (see the classical papers \cite{Tu48,Tu59}
and \cite{Go80}, and some more recent ones, as for instance \cite{CD02,Ca04,FK06,FK07}).
Hence, in this paper we focus on the structure of $\A_{g,b}$, while the new construction
of $\R_{g,b}$ mentioned above will be described in details in a forthcoming paper
\cite{BP12}.

The paper is structured as follows. In Section~\ref{main/sec} we give the basic
definitions, the statement of the main theorem, and the reduction of its proof to special
cases. Then, after having established some preliminary results in
Section~\ref{lemmas/sec}, we devote the subsequent Sections~\ref{order-pm/sec}, 
\ref{order-3m/sec}, \ref{order-2m/sec} to those special cases.


\section{Definitions and main theorem%
\label{main/sec}}

For a graph $\Gamma \in \G_{g,b}$, we call a {\sl free end} of $\Gamma$ any univalent
vertex, a {\sl terminal edge} of $\Gamma$ any of the $b$ edges connecting a free end to a
trivalent vertex, and an {\sl internal edge} of $\Gamma$ any of the $3g - 3 + b$ edges
connecting two (possibly coinciding) trivalent vertices.

First of all, we introduce $F$-moves on uni/trivalent graphs. These are well-known moves
that change the graph structure by acting on the internal edges. It is a folklore result
that $F$-moves, even in the most restrictive form given in
Definition~\ref{Fmove-edge/def}, suffice to relate any two graphs in $\G_{g,b}$ up to
graph isomorphisms (cf. \cite{HT80}). However, we need to consider an invariant version of
$F$-moves, in order to relate graph automorphisms. This makes things more involved.

\begin{definition}\label{Fmove-edge/def}
Let $\Gamma \in \Gamma_{g,b}$ and $e \in \Gamma$ be an internal edge with distinct ends.
Then, we call {\sl elementary (edge) $F$-move}, the modification $F_{e,e'}: \Gamma
\rightmove \Gamma'$ that makes $\Gamma$ into $\Gamma' \in \G_{g,b}$, by replacing $e$ with
a new internal edge $e'$ as in Figure~\ref{move-edge/fig}, while leaving the rest of the
graph unchanged. Clearly, the inverse modification $F_{e',e}: \Gamma' \rightmove \Gamma$
is an elementary (edge) $F$-move as well.
\end{definition}

\begin{Figure}[htb]{move-edge/fig}
\fig{}{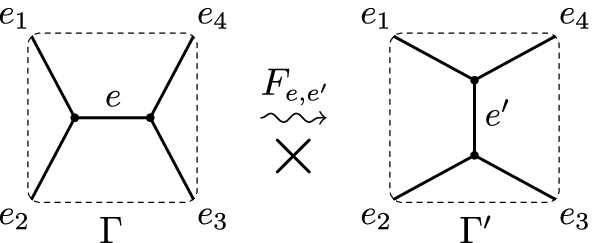}[-3pt]
{}{An edge $F$-move $F_{e,e'}$.}
\end{Figure}

According to the literature, here the letter $F$ stands for ``fusion'' and refers to the
fact that the $F$-move $F_{e,e'}$ can be thought as the contraction of the edge $e$ in
$\Gamma$, and the consequent fusion of its ends, followed by the inverse of a similar
contraction of the edge $e'$ in $\Gamma'$. The intermediate graph has a quadrivalent 
vertex in place of the two trivalent ends of the contracted edges, as indicated under
the arrow.

We warn the reader that, since the graphs in $\G_{g,b}$ are abstract ones, properties 
and constraints depending on the planar representation (of portions) of them in the 
figures do not have any significance.

In particular, the coupling $(e_1 e_2)(e_3 e_4)$ of the (possibly not all distinct) edges
$e_1, \dots, e_4$ determined by the two ends of $e$ in $\Gamma$, which in
Figure~\ref{move-edge/fig} is replaced by the coupling $(e_1 e_4)(e_2 e_3)$ determined by
the two ends of $e'$ in $\Gamma'$, could be replaced by the coupling $(e_1 e_3)(e_2 e_4)$
as well. In other words, there are exactly two possible ways to perform the $F$-move
$F_{e,e'}$ at the given edge $e$ of $\Gamma$, corresponding to the two possible changes of
coupling, depending on the structure of $\Gamma'$ at the edge $e'$:
$$(e_1 e_2)(e_3 e_4) \mathrel{\buildrel \textstyle F_{e,e'} \over \longrightwavearrow} 
(e_1 e_4)(e_2 e_3)
\quad \text{and} \quad
(e_1 e_2)(e_3 e_4) \mathrel{\buildrel \textstyle F_{e,e'} \over \longrightwavearrow} 
(e_1 e_3)(e_2 e_4).$$

From a different perspective, the elementary $F$-move described in
Figure~\ref{move-edge/fig}, can be interpreted as the replacement of a uni/trivalent
subtree $T \subset \Gamma$ with a different uni/trivalent subtree $T' \subset \Gamma'$
having the same univalent vertices. Namely, $T$ consists of the edges $e_1, \dots, e_4$
and $e$, while $T'$ consists of the edges $e_1, \dots, e_4$ and $e'$. This suggests the
following generalization of the notion of elementary $F$-move.
 
\begin{definition}\label{Fmove-tree/def}
Let $\Gamma \in \G_{g,b}$ and $T \subset G$ be a uni/trivalent subtree with $m \geq 4$
free ends. Then, we call {\sl elementary (tree) $F$-move}, the modification $F_{T,T'}:
\Gamma \rightmove \Gamma'$ that makes $\Gamma$ into $\Gamma' \in \G_{g,b}$, by replacing
$T$ with a different uni/trivalent tree $T' \subset \Gamma'$ having the same free ends of
$T$, while leaving the rest of the graph unchanged (see Figure~\ref{move-tree/fig} for an
example with $m = 5$). Also in this case, the inverse modification $F_{T',T}: \Gamma'
\rightmove \Gamma$ is an elementary (tree) $F$-move as well.
\end{definition}

\begin{Figure}[htb]{move-tree/fig}
\fig{}{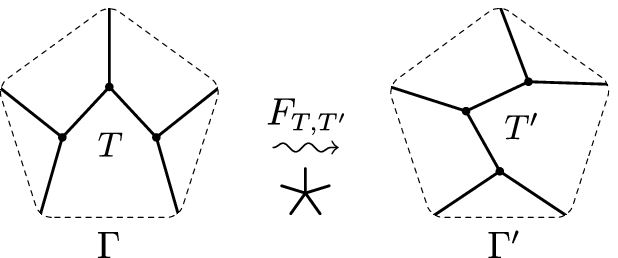}[-3pt]
{}{A tree $F$-move $F_{T,T'}$.}
\end{Figure}

We note that the $F$-move $F_{T,T'}$ is again a kind of fusion move. In fact, it can be
thought as the contraction of the subtree consisting of all the internal edges of $T
\subset \Gamma$ to a single $m$-valent vertex, and the consequent fusion of all the
trivalent vertices of $T$, followed by the inverse of a similar contraction for $T'
\subset \Gamma'$.

If the subtree $T \subset \Gamma$ has the vertices $v_1, \dots, v_m$ as its free ends,
then all the possible ways to perform the $F$-moves $F_{T,T'}$ on $\Gamma$ corresponds to
all the complete iterated couplings of the (unordered) set $\{v_1, \dots ,v_m\}$ induced 
by $T'$, except the one induced by $T$. In particular, according to the above discussion, 
we have 2 ways for $m = 4$.

Actually, any tree $F$-move $F_{T,T'}$ can be realized by a suitable sequence of edge 
$F$-moves performed on the edges of $T$ (and of the corresponding subtrees in the 
resulting graphs). Yet, it makes sense to consider $F_{T,T'}$ as a unique move, and we 
will see shortly why (cf. discussion about the $F$-move in Figure~\ref{move-inv/fig} 
below).

\begin{definition}\label{Fmove/def}
Let $\Gamma \in \G_{g,b}$ and $\T = \{T_1, \dots, T_k\}$ be a family of uni/trivalent
subtrees of $\Gamma$ with pairwise disjoint interiors (that is, two of them possibly share
only some common free ends). Assuming that each $T_i$ has at least 4 free ends, we call
{\sl $F$-move} any modification $F_{\T,\T'}: \Gamma \rightmove \Gamma'$ given by the
simultaneous application of certain elementary moves $F_{T_1,T'_1}, \dots, F_{T_k,T'_k}$
to $\Gamma$. Of course, for $k = 1$ we have an elementary $F$-move.
\end{definition}

For any graph $\Gamma \in \G_{g,b}$ we denote by $\A(\Gamma)$ the group of all
the automorphisms of $\Gamma$ as a non-oriented graph, including those that permute the
free ends. Moreover, we denote by $\A_{g,b} = \cup_{\Gamma \in \G_{g,b}}\A(\Gamma)$ the 
set of all the automorphisms of graphs in $\G_{g,b}$.

\medskip

Now, assume we are given an automorphism $\phi \in \A(\Gamma)$. If $F_{\T,\T'}: \Gamma
\rightmove \Gamma'$ is an $F$-move with $\T = \{T_1, \dots, T_k\}$ a $\phi$-invariant
family of subtrees of $\Gamma$, then the restriction of $\phi$ to $\Gamma - \cup_{i} \Int
T_i$ induces a graph automorphism $\psi$ of $\Gamma' - \cup_i\Int T'_i$, under\break the
canonical isomorphism $\Gamma - \cup_i \Int T_i \cong \Gamma' - \cup_i \Int T'_i$ of
graphs given by $F_{\T,\T'}$. In general, such $\psi$ does not extend to an automorphism
$\phi' \in \A(\Gamma')$, but if it does then the extension $\phi': \Gamma' \to \Gamma'$
can be easily seen to be unique, by taking into account that $\Gamma - \cup_i \Int T_i$
contains all the free ends of all the subtrees $T_i$. When $\phi'$ exists, we say that it
is induced by $\phi$ through $F_{\T,\T'}$.

\begin{definition}\label{Fmove-inv/def}
Given an automorphism $\phi: \Gamma \to \Gamma$, we say that an $F$-move $F_{\T,\T'}:
\Gamma \rightmove \Gamma'$ is {\sl $\phi$-invariant} if $\T$ is a $\phi$-invariant
family of subtrees and $\phi$ induces an automorphism $\phi': \Gamma' \to \Gamma'$ through
$F_{\T,\T'}$ as discussed above. In this case, we write $F_{\T,\T'}: \phi \rightmove
\phi'$. Such relation between automorphisms is symmetric, in the sense that $F_{\T',\T}:
\phi' \rightmove \phi$ holds as well, being $F_{\T',\T}: \Gamma' \rightmove \Gamma$ a
$\phi'$-invariant $F$- move. But it is neither reflexive nor transitive. Then, we call
{\sl F-equivalence} the generated equivalence relation on the set $\A_{g,b}$.
\end{definition}

We observe that the elementary $F$-move in Figure~\ref{move-edge/fig} is $\phi$-invariant 
with respect to any automorphism $\phi \in \A(\Gamma)$ that acts on the depicted 
portion of $\Gamma$ as any planar symmetry (horizontal, vertical or central). On the 
contrary, it is not $\phi$-invariant for a $\phi$ that switches $e_1$ and $e_2$ while 
leaving $e_3$ and $e_4$ fixed.

\begin{Figure}[htb]{move-inv/fig}
\fig{}{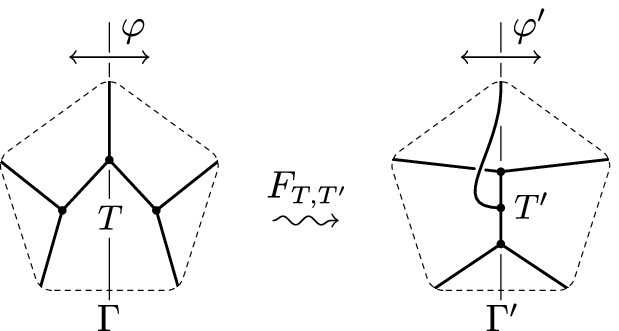}[-3pt]
{}{A $\phi$-invariant $F$-move.}
\end{Figure}

Figure~\ref{move-inv/fig} shows a $\phi$-invariant elementary $F$-move $F_{T,T'}$, where 
$\phi \in \A(\Gamma)$ is any automorphism that acts on $T$ as the simmetry with respect to 
the vertical axis. As we said above, $F_{T,T'}$ could be realized by a sequence of 
elementary edge $F$-moves, but this is not true if we insist that each single edge 
$F$-move is $\phi$-invariant. Finally, notice that instead the $F$-move described in 
Figure~\ref{move-tree/fig} is not $\phi$-invariant (for the same automorphism $\phi$).

\begin{remark}\label{slidings/rem}
In the following, we will essentially use only $\phi$-invariant $F$-moves $F_{\T,\T'}$ 
consisting of simultaneous elementary edge $F$-moves $F_{e,e'}$ or tree $F$-moves 
$F_{T,T'}$ like the one in Figure~\ref{move-inv/fig}. We emphasize once again that the 
$\phi$-invariance of $F$-moves $F_{\T,\T'}$ of does not imply the $\phi$-invariance of 
each single elementary $F$-move, being these possibly permuted by $\phi$. In many cases,
it will be convenient to think of those elementary moves as the edges slidings shown in 
Figure~\ref{move-3d/fig}. Here, the edges $f,f_1$ and $f_2$ are the original ones of 
$\Gamma$, whose ends in the figure are slided to get their new positions in $\Gamma'$ as 
indicated by the arrows, while the rest of the graph is fixed.
\end{remark}

\begin{Figure}[htb]{move-3d/fig}
\fig{}{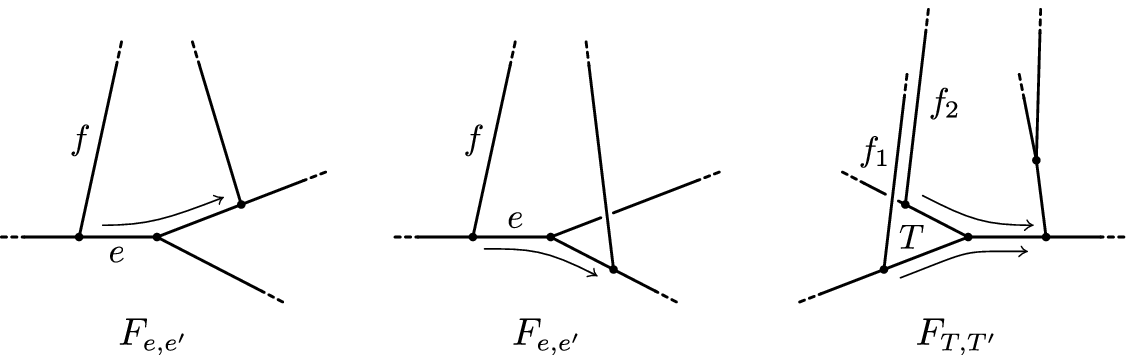}[-3pt]
{}{$F$-moves as edge slidings.}
\end{Figure}

At this point, in order to state our main theorem, we are left to introduce the elementary
automorphisms, which will generate all the automorphisms in $\A_{g,b}$ up to
$F$-equivalence.

\begin{definition}\label{switches/def}
By an {\sl elementary automorphism} of a graph $\Gamma \in \G_{g,b}$ we mean any
automorphism $S_{e_1,e_2}: \Gamma \to \Gamma$ interchanging two adjacent edges $e_1$ and
$e_2$ of $\Gamma$, while fixing all the rest of the graph. We call such an automorphism
$S_{e_1,e_2} \in \A(\Gamma)$ a {\sl terminal switch} or an {\sl internal switch}, 
depending on the fact that $e_1$ and $e_2$ are terminal or internal edges (cf.
Figure~\ref{switches/fig}).
\end{definition}

\begin{Figure}[htb]{switches/fig}
\fig{}{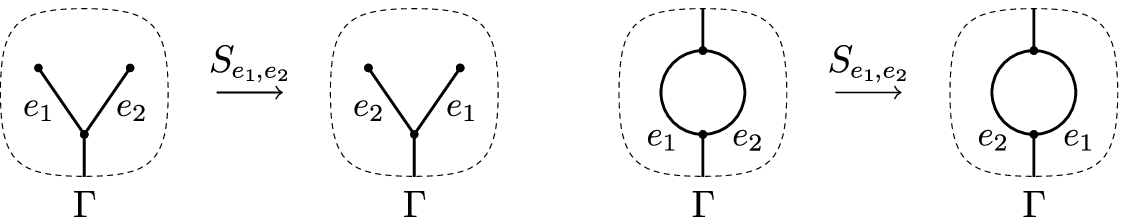}[-3pt]
{}{A terminal switch (left side) and an internal switch (right side).}
\end{Figure}

\begin{remark}
An automorphism in $\Gamma \in \G_{g,b}$ is uniquely determined up to internal switches by
its action on the vertices of $\Gamma$. Moreover, it is uniquely determined up to
(terminal and internal) switches by its action on the trivalent vertices of $\Gamma$.
\end{remark}

To simplify our claims, we provide the next definition.

\begin{definition}\label{Egn/def}
For any $g$ and $b$, we denote by $\E_{g,b} \subset \A_{g,b}$ the smallest subset that
contains all the elementary automorphisms (internal and terminal switches) in $\A_{g,b}$
and is closed with respect to composition and $F$-equivalence of automorphisms.\hfill
\end{definition}

Then, our main result can be stated as follows. 

\begin{theorem}\label{main/thm}
For any $g$ and $b$, we have $\E_{g,b} = \A_{n,b}$.
\end{theorem}

\begin{proof}
We have to show that any element $\phi \in \A_{g,b}$, that is any automorphism $\phi:
\Gamma \to \Gamma$ of a graph $\Gamma \in \G_{g,b}$, actually belongs to $\phi \in
\E_{g,b}$. By primary decomposition of the cyclic subgroup $\langle \phi \rangle \subset
\A(\Gamma)$, we can reduce ourselves to the special case when the order $\ord(\phi)$, that
is the cardinality of $\langle\phi\rangle$, is a prime power. Then, the subcases when
$\ord(\phi) = p^m$ with $p$ prime $> 3$ or $\ord(\phi) = 3^m$ are reduced to the case 
$\ord(\phi) = 2$ in Sections~\ref{order-pm/sec} and \ref{order-3m/sec} respectively. 
Finally, the case of $\ord(\phi) = 2^m$ is proved in Section~\ref{order-2m/sec}.
\end{proof}


\section{Some preliminary results%
\label{lemmas/sec}}

Let $\phi: \Gamma \to \Gamma$ be an automorphism of a uni/trivalent graph $\Gamma$.
Two elements (vertices or edges) or subgraphs $x$ and $y$ of $\Gamma$ are said to be {\sl 
$\phi$-equivalent} if $y = \phi^i(x)$ for some $i \geq 0$. In this case we write $x 
\cong_\phi y$.

We denote by $\ord(\phi)$ the order of $\phi$ in $\A(\Gamma)$ (as an automorphism of a
non-oriented graph), that is the cardinality of $\langle\phi\rangle \subset \A(\Gamma)$.
For a vertex $v$ of $\Gamma$, by the {\sl order of $v$ with respect to $\phi$}, we mean as
usual the cardinality $\ord_\phi(v)$ of its $\phi$-orbit $\Orb_\phi(v) = \{\phi^k(v) \,|\,
0 \leq k < \ord(\phi)\}$. On the contrary, for an edge $e$ of $\Gamma$ with (possibly
coinciding) ends $v$ and $w$, we call {\sl order of $e$ with respect to $\phi$} the number
$\ord_\phi(e) = \LCM(\ord_\phi(v), \ord_\phi(w))$, which is equal to either the
cardinality of the $\phi$-orbit $\Orb_\phi(e) = \{\phi^k(e) \,|\, 0 \leq k < \ord(\phi)\}$
or its double (if $v \neq w$ and $(v,w) \cong_\phi (w,v)$ as ordered pairs). This amounts
to consider $e$ as an oriented edge, except the case when it is a loop.

\begin{lemma}\label{ordinv/thm}
If $\phi$ and $\phi'$ are $F$-equivalent in $\A_{g,b}$, then $\ord(\phi) = \ord(\phi')$.
\end{lemma}

\begin{proof}
It suffices to consider the case when $\phi$ and $\phi'$ are related by a single $F$-move
$F_{\T,\T'}: \phi \rightmove \phi'$, with $\T = \{T_1, \dots, T_k\}$ and $\T' = \{T'_1,
\dots, T'_k\}$. Then, $\ord(\phi)$ coincides with the order of its restriction to $\Gamma
- \cup_i\, \Int T_i$, by the uniqueness of extension to the subtrees $T_i$. Similarly,
$\ord(\phi')$ coincides with the order of its restriction to $\Gamma' - \cup_i\, \Int
T'_i$. Since those restrictions coincide under the canonical isomorphism $\Gamma - \cup_i
\Int T_i \cong \Gamma' - \cup_i \Int T'_i$, we have $\ord(\phi) = \ord(\phi')$.
\end{proof}

\begin{lemma}\label{orders/thm}
Let $\phi: \Gamma \to \Gamma$ be an automorphism of a (possibly disconnected)
uni/trivalent graph $\Gamma$. If $v$ is a vertex of $\Gamma$ and $e$ is an edge of
$\Gamma$ having $v$ as an end, then we have one of the following: 
\begin{itemize}
\item[1)] $\ord_\phi(e) =
\ord_\phi(v)$; 
\item[2)] $\ord_\phi(e) = 2 \ord_\phi(v)$, which happens if and only if $v$ is a
trivalent vertex, with three distinct edges $e,e'$ and $e''$ exiting from it, such that
$e' = \phi^{\ord_\phi(v)}(e)$ and $\ord_\phi(e'') = \ord_\phi(v)$; 
\item[3)] $\ord_\phi(e) = 3
\ord_\phi(v)$, which happens if and only if $v$ is a trivalent vertex, with three distinct
edges $e,e'$ and $e''$ exiting from it, such that $e' = \phi^{\ord_\phi(v)}(e)$ and $e'' =
\phi^{2 \ord_\phi(v)}(e)$. 
\end{itemize}
Moreover, if $\Gamma$ is connected then the set of all orders of edges is $\{m, 2 m,
\dots, 2^k m\}$ for some $m \geq 1$ and $k \geq 0$ such that $\ord(\phi) = 2^k m$, while
the set of all orders of vertices is $\{m, 2 m, \dots, 2^h m, 2^{j_1} m/3, \dots,
2^{j_\ell} m/3\}$ with $k-1 \leq h \leq k$, $0 \leq l \leq k$, $0 \leq j_1 < \dots <
j_\ell \leq k$, and $m$ multiple of $3$ if $\ell > 0$.
\end{lemma}

\begin{proof}
By definition of order for edges, we immediately have that $\ord_\phi(e)$ is a multiple of
$\ord_\phi(v)$. Moreover, since $\phi^{i \ord_\phi(v)}(e)$ is an edge exiting from
$\phi^{i \ord_\phi(v)}(v) = v$ for any $i \geq 1$, the only possible cases are
$\ord_\phi(e) = \ord_\phi(v), 2 \ord_\phi(e), 3 \ord_\phi(v)$ and these are verified
according to the conditions given in the statement. As a consequence, if the orders of two
adjacent edges are different, then one of them is the double of the other, being those
edges like $e$ and $e''$ in point~2. For $\Gamma$ connected, this implies that the set of
orders of the edges is $\{m, 2 m, \dots, 2^k m\}$ for some $m \geq 1$ and $k \geq 0$. In
order to get the set of orders of vertices, it suffices to observe that if $e$ and $e'$
are two edges exiting from $v$, such that $\ord_\phi(e') = 2 \ord_\phi(e)$, then
$\ord_\phi(v) = \ord_\phi(e)$, which gives the orders $m, 2 m, \dots, 2^h m$ with $k-1
\leq h \leq k$. On the other hand, tripodes of edges whose order has the form $2^j m/3$
may appear, like in point~3, for any $j = 0, \dots, k$.
\end{proof}

By a {\sl path} $\alpha$ between the (possibly coinciding) vertices $v$ and $w$ in a graph
$\Gamma$ we mean a (possibly non-simple) chain of edges having $v$ and $w$ as its ends. We
call the number $\len(\alpha)$ of (non-necessarily distinct) edges in the chain the {\sl
length} of $\alpha$. Moreover, we denote by $\bar \alpha$ the reversed path.

\medskip

The next four lemmas concern minimal paths between vertices in a given $\phi$-orbit.

\begin{lemma}\label{paths-disjoint/thm}
Let $\phi: \Gamma \to \Gamma$ be an automorphism of a (possibly disconnected)
uni/trivalent graph $\Gamma$. Given any (simple) path $\alpha \subset \Gamma$ of minimal
length among all paths joining any two different vertices of a given $\phi$-orbit, let $v$
and $v'$ the ends of $\alpha$ in that $\phi$-orbit. If $i \neq j$ and the four points
$\phi^i(v), \phi^i(v'), \phi^j(v)$ and $\phi^j(v')$ are all distinct, then the paths
$\phi^i(\alpha)$ and $\phi^j(\alpha)$ are disjoint.
\end{lemma}

\begin{proof}
\begin{Figure}[b]{paths-disjoint-pf/fig}
\vskip-3pt
\fig{}{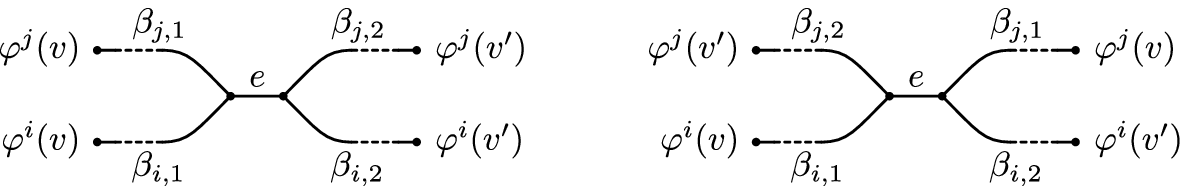}[-3pt]
{}{No common edge $e$ between $\phi^i(\alpha)$ and $\phi^j(\alpha)$ with four distinct 
   ends.}
\end{Figure}
Let $\alpha$ be a path as in the statement. By contradition, assume $\phi^i(\alpha)$ and
$\phi^j(\alpha)$ are not disjoint for some $i \neq j$ satisfying the required condition.
Then, they share at least a common internal edge $e$, and we have one of the two
situations depicted in Figure~\ref{paths-disjoint-pf/fig}, depending on the fact that $e$
is traversed in the same direction or not when going from $\phi^i(v)$ to $\phi^i(v')$
along $\phi^i(\alpha)$ and from $\phi^j(v)$ to $\phi^j(v')$ along $\phi^j(\alpha)$. Notice
that the paths $\beta_{i,1}$ and $\beta_{i,2}$ forming $\phi^i(\alpha) - \Int(e)$ are not
necessarily disjoint from the paths $\beta_{j,1}$ and $\beta_{j,2}$ forming
$\phi^j(\alpha) - \Int(e)$. In any case, we find a path shorter than $\alpha$ between
different vertices in the given $\phi$-orbit. Namely, such path is either $\beta_{i,1}
\bar \beta_{j,1}$ or $\beta_{i,2} \bar \beta_{j,2}$ in the left side case, while it is
either $\beta_{i,1} \bar \beta_{j,2}$ or $\beta_{i,2} \bar \beta_{j,1}$ in the right side
case.
\end{proof}

\begin{lemma}\label{paths-adjacent/thm}
Let $\phi: \Gamma \to \Gamma$, $\alpha \subset \Gamma$, $v$ and $v'$, be as in
Lemma~\ref{paths-disjoint/thm}. If $i \neq j$ with $\phi^j(v) = \phi^i(v')$ and
$\phi^j(v') \neq \phi^i(v)$, then the union $\phi^i(\alpha) \cup \phi^j(\alpha)$ has the
structure shown in Figure~\ref{paths-adjacent/fig}, where: 1)~$k \geq 1$; 2)~the paths
$\delta_{i,1}, \dots, \delta_{i,k} \subset \phi^i(\alpha)$ are disjoint from the paths
$\delta_{j,1}, \dots, \delta_{j,k} \subset \phi^j(\alpha)$; 3)~$\len(\delta_{i,h}) =
\len(\delta_{j,k-h+1}) \geq 1$ for every $h = 1, \dots, k$; 4)~$\len(\delta_{i,1}) =
\len(\delta_{j,k}) \geq \len(\alpha)/2$; 5) all the common paths shared by
$\phi^i(\alpha)$ and $\phi^j(\alpha)$ have length $\geq 1$, except the terminal one ending
at $\phi^j(v) = \phi^i(v')$, which can have length zero.
\end{lemma}

\begin{Figure}[htb]{paths-adjacent/fig}
\vskip-6pt
\fig{}{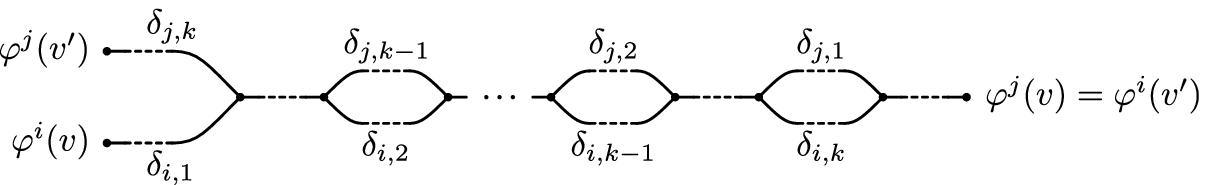}[-3pt]
{}{$\phi^i(\alpha) \cup\phi^j(\alpha)$ 
   for $\phi^j(v) = \phi^i(v')$ and $\phi^j(v') \neq \phi^i(v)$.}
\end{Figure}

\begin{proof}
Look at Figure~\ref{paths-disjoint-pf/fig} assuming $\phi^j(v) = \phi^i(v')$ and
$\phi^j(v') \neq \phi^i(v)$. By a similar argument as in the proof of the previous lemma,
the minimality of $\alpha$ implies that $\phi^i(\alpha)$ and $\phi^j(\alpha)$ cannot share
an edge $e$ like in the left side of the figure, while if $\phi^i(\alpha)$ and
$\phi^j(\alpha)$ share an edge $e$ like in the right side of the figure then
$\len(\beta_{i,1}) = \len(\beta_{j,2}) \geq \len(\alpha)/2$. This gives properties 1 and
4.

\begin{Figure}[htb]{paths-adjacent-pf/fig}
\fig{}{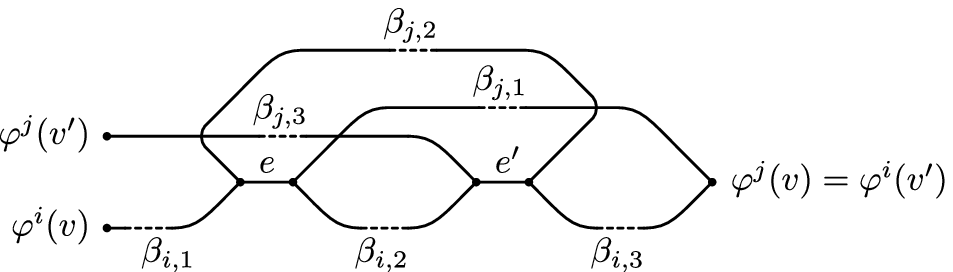}[-3pt]
{}{Common edges $e$ and $e'$ cannot occur in the same order along the paths 
$\phi^i(\alpha)$ and $\phi^j(\alpha)$ in Lemma~\ref{paths-adjacent/thm}.}
\end{Figure}

For property 2, it suffices to observe that, if $e$ and $e'$ are two common edges of
$\phi^i(\alpha)$ and $\phi^j(\alpha)$, then they occur in opposite orders along the two
paths starting from $\phi^i(v)$ and $\phi^i(v')$ respectively. Indeed, in the contrary
case, we would have $\phi^i(\alpha) = \beta_{i,1} e \beta_{i,2} e' \beta_{i,3}$ and
$\phi^j(\alpha) = \beta_{j,1} \bar e \beta_{j,2} \bar e' \beta_{j,3}$ like in
Figure~\ref{paths-adjacent-pf/fig}, where $\beta_i$'s and the $\beta_j$'s are possibly
empty and non-necessarily disjoint. Then, the minimality of $\phi^i(\alpha)$ would imply
$\len(\beta_{i,2}) + 2 \leq \len(\beta_{j,2})$, while the minimality of $\phi^j(\alpha)$
would imply $\len(\beta_{j,2}) + 2 \leq \len(\beta_{i,2})$, which would be absurd.

At this point, property~3 and 5 immediately follow from the minimality of $\phi^i(\alpha)$
and $\phi^j(\alpha)$ and from the trivalency of the vertices respectively.
\end{proof}

\begin{lemma}\label{paths-diagonal/thm}
Let $\phi: \Gamma \to \Gamma$, $\alpha \subset \Gamma$, $v$ and $v'$, be as in
Lemma~\ref{paths-disjoint/thm}. If $i \neq j$ with $\phi^j(v) = \phi^i(v')$ and
$\phi^j(v') = \phi^i(v)$, then the union $\phi^i(\alpha) \cup \phi^j(\alpha)$ has the
structure shown in Figure~\ref{paths-diagonal/fig}, where: 1)~$k \geq 0$; 2)~the paths
$\delta_{i,1}, \dots, \delta_{i,k} \subset \phi^i(\alpha)$ are disjoint from the paths
$\delta_{j,1}, \dots, \delta_{j,k} \subset \phi^j(\alpha)$; 3)~$\len(\delta_{i,h}) =
\len(\delta_{j,k-h+1}) \geq 1$ for every $h = 1, \dots, k$; 4) all the common paths shared
by $\phi^i(\alpha)$ and $\phi^j(\alpha)$ have length $\geq 1$, except both the terminal
ones that can have length zero.

Moreover, this may happen only when $\ord(\phi)$ is even.
\end{lemma}

\begin{Figure}[htb]{paths-diagonal/fig}
\vskip-3pt
\fig{}{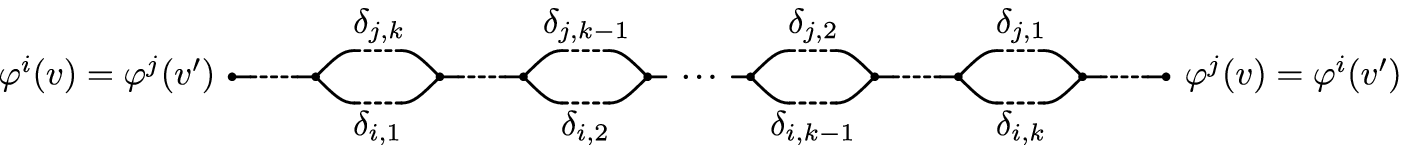}[-3pt]
{}{$\phi^i(\alpha) \cup\phi^j(\alpha)$ 
   for $\phi^i(v) = \phi^j(v')$ and $\phi^j(v) = \phi^i(v')$.}
\end{Figure}

\begin{proof}
The same arguments in the proof of Lemma~\ref{paths-adjacent/thm} still work here, to give
the first part of the statement. In fact, in that proof the assumption that $\phi^j(v')
\neq \phi^i(v)$ was only used to obtain $k \geq 1$ and property 4 (in the statement of
that lemma). 

For the last sentence, we observe that the conditions $\phi^j(v) = \phi^i(v')$ and
$\phi^j(v') = \phi^i(v)$ imply that $\phi^{j-i}$ swaps $v$ and $v'$, which in turn implies
that $\ord(\phi^{j-i})$, and hence $\ord(\phi)$ as well, is even.
\end{proof}

\begin{lemma}\label{paths-doubled/thm}
Let $\phi: \Gamma \to \Gamma$, $\alpha \subset \Gamma$, $v$ and $v'$, be as in
Lemma~\ref{paths-disjoint/thm}. If $i \neq j$ with $\phi^i(v) = \phi^j(v)$ and $\phi^i(v')
= \phi^j(v')$, then the union $\phi^i(\alpha) \cup \phi^j(\alpha)$ has the structure shown
in Figure~\ref{paths-doubled/fig}, where: 1)~$k \geq 0$; 2)~the paths $\delta_{i,1},
\dots, \delta_{i,k} \subset \phi^i(\alpha)$ are disjoint from the paths $\delta_{j,1},
\dots, \delta_{j,k} \subset \phi^j(\alpha)$; 3)~$\len(\delta_{i,h}) = \len(\delta_{j,h})
\geq 1$ for every $h = 1, \dots, k$; 4) all the common paths shared by $\phi^i(\alpha)$
and $\phi^j(\alpha)$ have length $\geq 1$, except both the terminal ones that can have 
length zero. 

Moreover, apart for the trivial case of $\phi^i(\alpha) = \phi^j(\alpha)$, 
this may happen only when $\ord(\phi)$ is either even or an odd multiple of 3, being in 
the latter case $\phi^i(\alpha) = \delta_{i,1}$ and $\phi^j(\alpha) = \delta_{j,1}$.
\end{lemma}

\begin{Figure}[htb]{paths-doubled/fig}
\fig{}{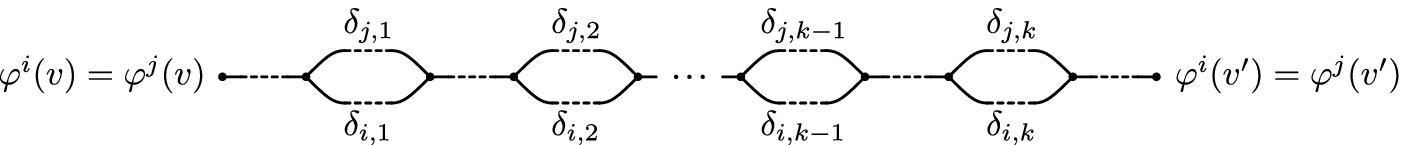}[-3pt]
{}{$\phi^i(\alpha) \cup\phi^j(\alpha)$ 
   for $\phi^i(v) = \phi^j(v)$ and $\phi^i(v') = \phi^j(v')$.}
\end{Figure}

\begin{proof}
The proof of the first part of the statement is completely analogous to that of the 
previous lemma, except for the orderings of the $\delta_i$'s along $\phi^i(\alpha)$ and of 
the $\delta_j$'s along $\phi^j(\alpha)$, which now coincide. 

To prove the second part, assume that $\phi^i(\alpha) \neq \phi^j(\alpha)$ and
$\ord(\phi)$ is odd. Then, we have $\phi^i(\alpha) = \delta_{i,1}$ and $\phi^j(\alpha) =
\delta_{j,1}$. Other\-wise, $\phi^{j-i}$ would fix some edge of $\phi^i(\alpha) \cap
\phi^j(\alpha)$ and $\ord(\phi^{j-i})$ should be even by Lemma~\ref{orders/thm}, in
contrast with the oddness of $\ord(\phi)$. For the same reason $\phi^{j-i}$ cannot swap
$\phi^i(\alpha)$ and $\phi^j(\alpha)$, then it cyclically permutes the three paths
$\phi^i(\alpha)$, $\phi^j(\alpha)$ and $\phi^{2j-i}(\alpha)$, which meet each other only
at their common ends $\phi^i(v)$ and $\phi^i(v')$. This implies that $\ord(\phi^{j-i})$,
and hence $\ord(\phi)$ as well, is a multiple of 3.
\end{proof}


\section{The case of order $p^m$ with $p$ prime greater than $3$%
\label{order-pm/sec}}

The aim of this section is to show that any automorphism $\phi \in \A_{g,b}$ with order
$\ord(\phi) = p^m$ for a prime $p > 3$ is $F$-equivalent to the composition of two
automorphisms of order $2$. 

Actually, we will prove this fact under the weaker assumption that $n = \ord(\phi)$ is a
multiple of neither $2$ nor $3$. In fact, all the arguments are only based on the 
following special properties, which hold for any automorphism $\phi: \Gamma \to \Gamma$
of a (possibly disconnected) uni/trivalent graph $\Gamma$, having such an order $n$:
\begin{itemize}
\item[1)] $\ord_\phi(v) = \ord_\phi(e) = n$ for any vertex $v$ and edge $e$ of $\Gamma$, 
thanks to Lemma~\ref{orders/thm};
\item[2)] the situations of Lemmas~\ref{paths-diagonal/thm} and \ref{paths-doubled/thm} 
cannot occur.
\end{itemize}
\medskip

\begin{lemma}\label{cycle-p/thm}
Let $\phi: \Gamma \to \Gamma$ be an automorphism of a (possibly disconnected)
uni/trivalent graph $\Gamma$, whose order $n = \ord(\phi)$ is a multiple of neither $2$
nor $3$.\break If $\alpha \subset \Gamma$ is a (simple) path of minimal length among all
the paths joining any two distinct $\phi$-equi\-valent vertices, then $\cup_i\,
\phi^i(\alpha)$ is a disjoint union $C \sqcup \phi(C) \sqcup \dots \sqcup
\phi^{n/\ell - 1}(C)$ of $n/\ell$ simple cycles (possibly a single one, for $\ell = n$),
each given by a concatenation of $\ell$ images of $\alpha$ with $\ell$ a divisor of $n$
greater than $1$.\break In fact, there exists a positive integer $s = t\, n/\ell < n$
with $(t,\ell) = 1$, such that the cycle $C$ is given by $\alpha \,\phi^s(\alpha)
\cdots \phi^{(\ell - 1) s}(\alpha)$. Moreover, up to $F$-equivalence we can assume 
$\alpha$ to consist of a single edge $a$ (see Figure \ref{cycle/fig}, where $a_i$ stands
for $\phi^{is}(a)$ and a similar notation is adopted for the $v_i$'s as well).
\end{lemma}

\begin{Figure}[htb]{cycle/fig}
\fig{}{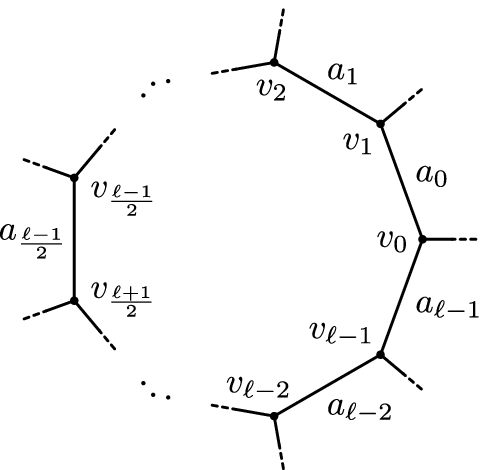}[-3pt]
{}{The form of the cycle $C$ in Lemma~\ref{cycle-p/thm}.}
\end{Figure}

\begin{proof}
Let $\alpha$ be a path as in the statement and let $v \neq v'$ be its ends. For any $i
\neq j \!\!\mod n$, let us consider the two paths $\phi^i(\alpha)$ and $\phi^j(\alpha)$,
and the four different situations described in Lemmas~\ref{paths-disjoint/thm} to
\ref{paths-doubled/thm}, which cover all the possibilities, thanks to the
$\phi$-equivalence of $v$ and $v'$.

The situations of Lemmas~\ref{paths-diagonal/thm} and \ref{paths-doubled/thm} cannot
occur, hence in particular $\phi^i(\alpha) \neq \phi^j(\alpha)$. Therefore,
$\phi^i(\alpha)$ and $\phi^j(\alpha)$ can have non-empty intersection only as in
Lemma~\ref{paths-adjacent/thm}. In this situation, if $\phi^i(u)$ with $u \in \alpha$ is
the common end of $\delta_{i,1}$ and $\delta_{j,k}$, then $\phi^j(u) \in \delta_{j,k}$
since $\len(\delta_{i,1}) = \len(\delta_{j,k}) \geq \len(\alpha)/2$. Actually, $\phi^j(u)$
has to coincide with either $\phi^j(v')$ or $\phi^i(u)$, otherwise the global minimality
of $\alpha$ would be contradicted.

Since $\phi^j(u) = \phi^i(u)$ is impossible, being $\ord_\phi(u) = n$, we are left with
the unique possibility $\phi^j(u) = \phi^j(v')$, that is $u = v'$. Then, we have
$\phi^i(\alpha) = \delta_{i,1}$ and $\phi^j(\alpha) = \delta_{j,k}$, which meet only at
their common end $\phi^j(v) = \phi^i(v')$. Moreover, such end cannot be shared by any
other $\phi^h(\alpha)$. Otherwise, either $\phi^i(\alpha)$ and $\phi^h(\alpha)$ or
$\phi^j(\alpha)$ and $\phi^h(\alpha)$ would be in the situation of
Lemma~\ref{paths-doubled/thm}.

Then, we can conclude in a straightforward way that $\cup_i\,\phi^i(\alpha)$ has the
stated form, with $s < n$ uniquely determined by $v' = \phi^s(v)$, $\ell$ the smallest
positive integer such that $\phi^{\ell s}(v) = v$, and $t = \ell\, s/n$ (this is an
integer since $\ord_\phi(v) = n$ and it is coprime with $\ell$ by the minimality of
$\ell\,$).

\begin{Figure}[htb]{cycle-pf/fig}
\fig{}{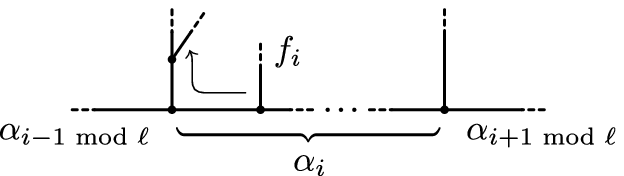}[-3pt]
{}{Reducing $\len(\alpha)$ in Lemma \ref{cycle-p/thm}.}
\end{Figure}

At this point, we are left to prove that the path $\alpha$ can be replaced by a single
edge up to $F$-equivalence. Proceeding by induction, it suffices to show how to reduce
$\len(\alpha)$ whenever this is greater than 1. Such reduction can be achieved by sliding
all the copies $f_i = \phi^{is}(f)$ of the unique edge $f$ attached at the first
intermediate vertex of $\alpha$ but not belonging to $\alpha$, so that their ends are
slided out of the paths $\alpha_i = \phi^{is}(\alpha)$, as indicated in Figure
\ref{cycle-pf/fig}. According to Remark \ref{slidings/rem}, those slidings correspond to a
$\phi$-invariant $F$-move, given by simultaneous elementary $F$-moves performed on the
first edges of all the $\alpha_i$'s, which are pairwise disjoint if $\len(\alpha) \geq 2$.
\end{proof}

\begin{Figure}[b]{onion/fig}
\fig{}{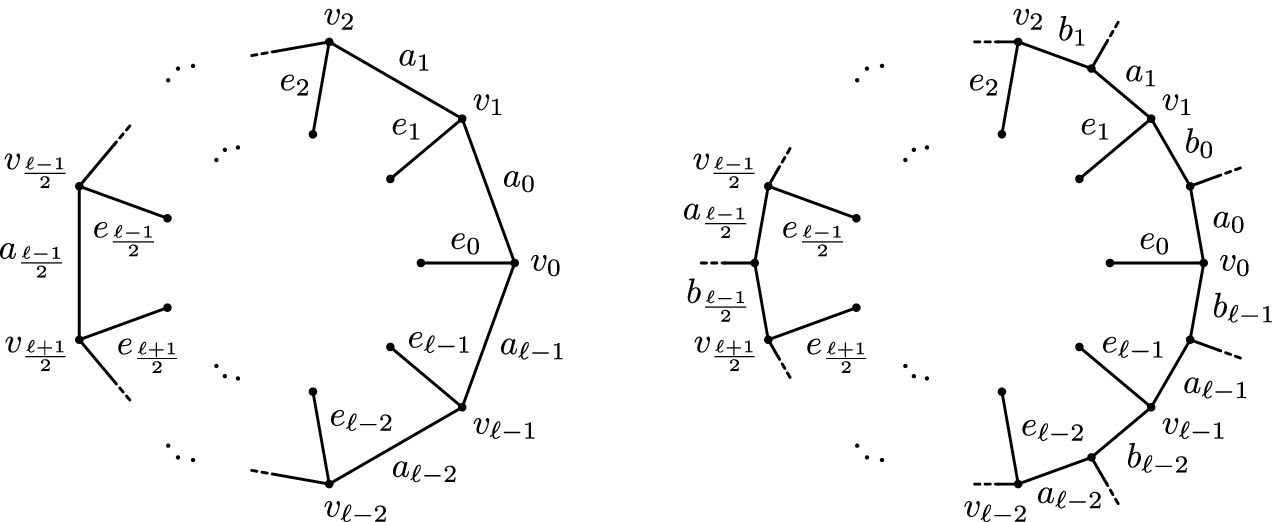}[-3pt]
{}{The two possible forms of the cycle $C$ in Lemma~\ref{step-p/thm}.}
\end{Figure}

\begin{lemma}\label{step-p/thm}
Let $\phi: \Gamma \to \Gamma$ be an automorphism of a (possibly disconnected)
uni/trivalent graph $\Gamma$, whose order $n = \ord(\phi)$ is a multiple of neither $2$
nor $3$.\break If $\alpha \subset \Gamma$ is a (simple) path of minimal length among all
paths joining any two distinct (hence disjoint, see proof) terminal edges in a given
$\phi$-orbit, then up to $F$-equivalence we can assume $\cup_i\,\phi^i(\alpha)$ to be
a disjoint union $C \sqcup \phi(C) \sqcup \dots \sqcup \phi^{n/\ell - 1}(C)$ of
$n/\ell$ cycles (possibly a single one, for $\ell = n$), with $\ell$ a divisor of $n$
greater than $1$ and $C = \alpha \,\phi^s(\alpha) \cdots \phi^{(\ell - 1) s}(\alpha)$ for
some positive integer $s = t\, n/\ell < n$ with $(t,\ell) = 1$, like in Lemma
\ref{cycle-p/thm}. Moreover, $\alpha$ can be assumed to
consist of either one edge $a$ or two edges $a$ and $b$ (see Figure \ref{onion/fig}, where
$a_i$ and $b_i$ stand for $\phi^{is}(a)$ and $\phi^{is}(b)$ respectively, and a similar
notation is adopted for the terminal edges $e_i$'s and the vertices $v_i$'s as well).
\end{lemma}

\begin{proof}
Let $\alpha$ be a path as in the statement, and $e \neq e'$ be the terminal edges it
joins. Then the ends of $\alpha$ coincide with the unique trivalent ends $v$ and $v'$ 
of $e$ and $e'$ respectively. Notice that $v \neq v'$, being $\ord_\phi(v) = \ord_\phi(e) 
= n$ by Lemma \ref{orders/thm}.

For any $i \neq j \!\!\mod n$, by arguing as in the proof of Lemma~\ref{cycle-p/thm}, we
can prove that $\phi^i(\alpha)$ and $\phi^j(\alpha)$ can meet only as in
Lemma~\ref{paths-adjacent/thm}, and that their common end $\phi^j(v) = \phi^i(v')$ cannot
be shared by any other $\phi^h(\alpha)$. The same Lemma \ref{paths-adjacent/thm} also
tells us that $\len(\delta_{i,1}) = \len(\delta_{j,k}) \geq \len(\alpha)/2$ (cf.
Figure~\ref{paths-adjacent/fig}). But here the equality cannot occur, otherwise
$\phi^{j-i}$ would fix the common end of $\delta_{i,1}$ and $\delta_{j,k}$, and it would
cyclically permute the three edges exiting from it, in contrast with the assumption that
$n$ is not a multiple of 3.

Then, we can conclude straightforwardly that $\cup_i\,\phi^i(\alpha)$ consists of the
disjoint union $C \sqcup \phi(C) \sqcup \dots \sqcup \phi^{n/\ell - 1}(C)$, with $s < n$ 
uniquely determined by $\phi^s(v) = v'$, $\ell$ the smallest positive integer such that
$\phi^{\ell s}(v) = v$, $t = s\, \ell/n$ (cf. proof of Lemma \ref{cycle-p/thm}), and $C =
\alpha \,\phi^s(\alpha) \cdots \phi^{(\ell - 1) s}(\alpha)$ having the form depicted in
Figure \ref{onion-pf1/fig}. Here, and in the next figure as well, $\alpha_i$ stands for
the path $\phi^{is}(\alpha)$, which is the part of $C$ in the corresponding circular
sector, and each arc represents a path of edges.

\begin{Figure}[htb]{onion-pf1/fig}
\vskip6pt
\fig{}{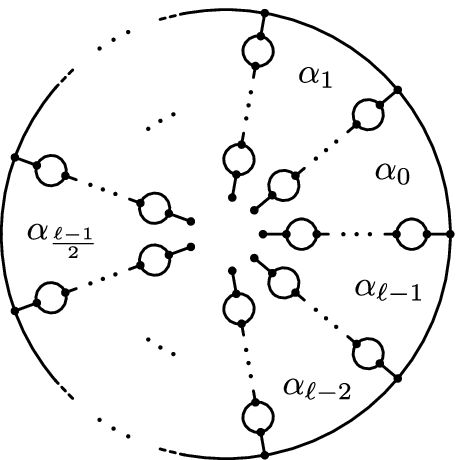}
{}{The starting configuration $C$ in the proof of Lemma \ref{step-p/thm}.}
\end{Figure}

In the very special case where the paths $\alpha_i$ consist of single edges, they can only
share their common ends. Hence, the configuration $C$ in Figure \ref{onion-pf1/fig}
coincides with that in the left side of Figure \ref{onion/fig}, and we are done.

Otherwise, by a sequence of $\phi$-invariant slidings (cf. Remark \ref{slidings/rem}) on
the edges attached to $C$ and to its copies, as indicated in Figure \ref{onion-pf2/fig}
for three of such edges, we can reduce all the arcs in Figure \ref{onion-pf1/fig} to
single edges, possibly except the ones forming the big circle. In the same way, when the
first edge of $\alpha_i$ coincides with the last edge of $\alpha_{i-1 \!\!\mod \ell}$, we
also fuse it with the edge $e_i = \phi^{is}(e)$.

\begin{Figure}[htb]{onion-pf2/fig}
\fig{}{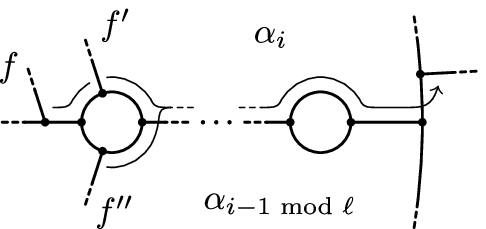}[-3pt]
{}{Simplifying the starting configuration.}
\end{Figure}

Then, we can eliminate all the resulting bigons, by a $\phi$-invariant $F$-move, which
atcs on $C$ by simultaneous elementary $F$-moves on the edges between any two consecutive
of them and between the most external ones and the big circle, as shown in Figure
\ref{onion-pf3/fig}.

\begin{Figure}[t]{onion-pf3/fig}
\fig{}{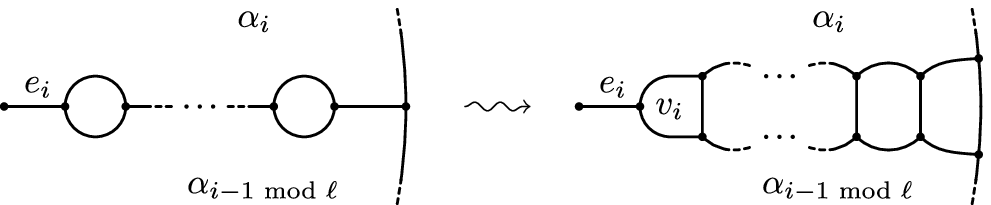}[-3pt]
{}{Eliminating the bigons by simultaneous edge $F$-moves.}
\end{Figure}

\begin{Figure}[htb]{onion-pf4/fig}
\fig{}{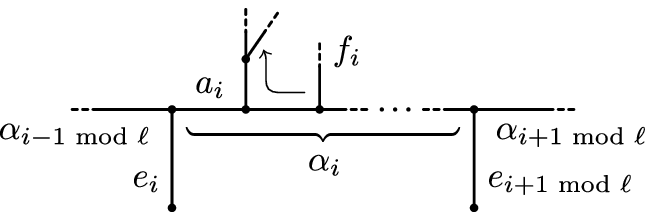}
{}{Reducing $\len(\alpha)$ in Lemma \ref{step-p/thm}.}
\end{Figure}

Finally, to get the configuration in the right side of Figure \ref{onion/fig}, it remains
to reduce $\alpha$ to a chain of only two edges $a$ and $b$. This can be done proceeding 
by induction, like in the last part of the proof of Lemma \ref{cycle-p/thm}. In Figure 
\ref{onion-pf4/fig} we see how to reduce $\len(\alpha)$ whenever this is greater than 2.
Here, $a_i = \phi^{is}(a)$ and $f_i = \phi^{is}(f)$ are respectively the images in 
$\alpha_i = \phi^{is}(\alpha)$ of the first edge $a$ of $\alpha$ and of the edge $f$ 
attached at the second intermediate vertex of $\alpha$ but not belonging to $\alpha$.
The desired reduction is achieved through an invariant $F$-move, which simultaneously 
slides all the $f_i$, in such a way that their ends pass from the $\alpha_i$'s to the 
first edges attached to them. Notice that, contrary to what we did in the proof of Lemma 
\ref{cycle-p/thm}, we cannot further reduce $\alpha$ to a single edge by a slide letting
even the last intermediate point of $\alpha$ pass to $e$.
\end{proof}

\begin{proposition}\label{onion-p/thm}
Let $\phi \in \A_{g,b}$ be a non-trivial automorphism, whose order $n = \ord(\phi) > 1$ is
a multiple of neither $2$ nor $3$. Then, $\phi$ is $F$-equivalent to the composition $\tau
\circ \sigma$ of two automorphisms $\sigma, \tau \in \A_{g,b}$ such that $\ord(\sigma) =
\ord(\tau) = 2$. Moreover, both $\sigma$ and $\tau$ can be assumed to fix an edge and
reverse an invariant edge.\hfill
\end{proposition}

\begin{proof}
Let $\phi: \Gamma \to \Gamma$ any automorphism in $\A_{n,g}$ as in the statement. We will
construct by recursion a sequence $\phi_i: \Gamma_i \to \Gamma_i$ of automorphisms in
$\A_{g,b}$ and a sequence of subgraphs $\Lambda_i \subset \Gamma_i$ with $i = 0, \dots, r$
and $r \geq 2$, having the following properties:
\begin{itemize}
\item[1)] $\phi_0 = \phi$ (hence $\Gamma_0 = \Gamma$), and $\phi_i$ is $F$-equivalent to 
$\phi_{i-1}$ for every $1 \leq i \leq r$;
\item[2)] $\emptyset = \Lambda_0 \varsubsetneq \Lambda_1 \varsubsetneq \dots \varsubsetneq 
\Lambda_r = \Gamma_r$, and $\Delta_i = \Cl(\Gamma_i - \Lambda_i)$ is a uni/trivalent graph 
whose intersection with $\Lambda_i$ is a single $\phi_i$-orbit $U_i \subset \Delta_i$ of 
its free ends for every $1 \leq i \leq r-1$;
\item[3)] $\Lambda_i$ is $\phi_i$-invariant and $\phi_{i|\Lambda_{i-1}} = 
\phi_{i-1|\Lambda_{i-1}}$ for every $1 \leq i \leq r$;
\item[4)] $\phi_{i|\Lambda_i}: \Lambda_i \to \Lambda_i$ is the composition $\tau_i \circ
\sigma_i$ of two automorphisms $\sigma_i: \Lambda_i \to \Lambda_i$ and $\tau_i: \Lambda_i 
\to \Lambda_i$ such that $\ord(\sigma_i) = \ord(\tau_i) = 2$, for every $1 \leq i \leq r$;
moreover, $\sigma_i$ fixes exactly one vertex $u_i$ in $U_i$ (hence $\tau_i$ fixes
$\phi_i^{(n+1)/2}(u_i) \in U_i$) for every $1 \leq i \leq r-1$, and both $\sigma_i$ and
$\tau_i$ fix an edge and reverse an invariant edge for every $2 \leq i \leq r$.
\end{itemize}

This will prove the proposition, since $\sigma = \sigma_r$ and $\tau = \tau_r$ satisfy 
all the required conditions, being $\tau \circ \sigma = \tau_r 
\circ \sigma_r = \phi_{r|\Lambda_r} = \phi_r$ $F$-equivalent to $\phi$ in $\A_{g,b}$.

To start the construction, we apply Lemma \ref{cycle-p/thm} to the automorphism $\phi$
and a minimal path $\alpha \subset \Gamma$ as in that statement (this exists since
$\Gamma$ is connected but not reduced to a single vertex and $\phi$ is non-trivial). As a
result, we get an automorphism $\phi_1: \Gamma_1 \to \Gamma_1$ in $\A_{g,b}$, which is
$F$-equivalent to $\phi_0 = \phi$ and such that the minimal path corresponding to $\alpha$
in $\Gamma_1$ consists of only one (internal) edge $a \subset \Gamma_1$. Then, we put
$\Lambda_1 = \cup_j\, \phi_1^j(a) \subset \Gamma_1$, which is $\phi_1$-invariant by
definition. Lemma \ref{cycle-p/thm} tells us that there exist a divisor $\ell > 1$ of $n$
and an integer $0 < t < l$ with $(t,\ell) = 1$, such that $\Lambda_1$ consists of the
disjoint union $C \sqcup \phi_1(C) \sqcup \dots \sqcup \phi_1^{n/\ell - 1}(C)$ of $n/\ell$
simple cycles (possibly a single one, for $\ell = n$), where $C = a \,\phi_1^s(a) \cdots
\phi_1^{(\ell - 1) s}(a)$\break with $s = t\, n/\ell$.

Let $v$ be the first end of the edge $a$, so that $\phi_1^s(v)$ is the second one. Then,
the set of vertices of $\Lambda_1$ is the $\phi_1$-orbit $V = \{v, \phi_1(v), \dots,
\phi_1^{n-1}(v)\}$, and any automorphism on $\Lambda_1$ is uniquely determined by its
action on $V$, since $n$ is odd and hence $\ell > 2$. Therefore, we can first define the
two involutions $\sigma_1$ and $\tau_1$ on $V$ in such a way that $\tau_1 \circ \sigma_1$
coincides with the restriction of $\phi_1$ to $V$, then check that they preserve
adjacency of vertices in $\Lambda_1$ and hence give well-defined automorphisms of 
$\Lambda_1$.

We define $\sigma_1$ and $\tau_1$ on $V$ by putting $\sigma_1(\phi_1^j(v)) =
\phi_1^{-j}(v)$ and $\tau_1(\phi_1^j(v)) = \phi_1^{1-j}(v)$, for $j = 0, \dots, n-1$.
Clearly, these are involutions and their composition $\tau_1 \circ \sigma_1$ acts on $V$
as $\phi_1$. Concerning the preservation of adjacency, it suffices to observe that two
elements $\phi_1^j(v)$ and $\phi_1^k(v)$ of $V$ are adjacent vertices in $\Lambda_1$ if
and only if $|j - k| = s \!\!\mod n$, and both $\sigma_1$ and $\tau_1$ preserve this
condition.

Finally, we note that $\Delta_1 = \Cl(\Gamma_1 - \Lambda_1)$ is a uni/trivalent graph and
$\Delta_1 \cap \Lambda_1$ is given by the $\phi_1$-orbit $U_1 = V$ of free ends of
$\Delta_1$, as required in point~2 above. Moreover, $\sigma_1$ fixes only the vertex $u_1 
= v$ of $U_1$ and reverses the edge $\phi_1^{(n+1)/2}(a)$, while $\tau_1$ reverses the 
edge $a$, as required in point~4 above.

Now, to realize the recursive step of the construction, assume we are given $\phi_{i-1}:
\Gamma_{i-1} \to \Gamma_{i-1}$, $u_{i-1} \in U_{i-1} \subset \Lambda_{i-1} \varsubsetneq
\Gamma_{i-1}$ and $\sigma_{i-1},\tau_{i-1}: \Lambda_{i-1} \to \Lambda_{i-1}$ satisfying
all the properties stated in the above points~1 to 4 for $i-1 < r$.\break Lemmas
\ref{ordinv/thm} and \ref{orders/thm} ensure that $\ord(\phi_{i-1}) = n$ and that the
cardinality of the $\phi_{i-1}$-orbit $U_{i-1}$ is equal to $n$ as well. Hence, also the
restriction of $\phi_{i-1}$ to the uni/tri\-valent graph $\Delta_{i-1} = \Cl(\Gamma_{i-1}
- \Lambda_{i-1})$ is an automorphism of the same order $n$. Furthermore, the terminal
edges of $\Delta_{i-1}$ ending at vertices in $U_{i-1}$ are all distinct. In fact, if two
of them would coincide, then their $\phi$-order would be even, which contradicts the
hypothesis that $n$ is odd.

If $\Delta_{i-1}$ does not contain any path joining different terminal edges ending at
vertices in $U_{i-1}$, then it consists of $n$ components each containing a single vertex
in $U_{i-1}$. Denoting by $C$ the component of $\Delta_{i-1}$ containing that vertex
$u_{i-1}$, we have the component decomposition $\Delta_{i-1} = C \sqcup \phi_{i-1}(C)
\sqcup \dots \sqcup \phi_{i-1}^{n-1}(C)$, with $\phi_{i-1}$ cyclically permuting the
components. In this case, we put $\phi_i = \phi_{i-1}$ and $\Lambda_i = \Gamma_i =
\Gamma_{i-1}$. Moverover, we define $\sigma_i$ and $\tau_i$ as the unique automorphisms of
$\Gamma_i$ extending $\sigma_{i-1}$ and $\tau_{i-1}$, in such a way that they act on each
$\phi_{i-1}^j(C)$ with $j = 0, \dots, n-1$, as $\phi_{i-1}^{-2j}$ and $\phi_{i-1}^{1-2j}$
respectively. A straightforward verification shows that such $\phi_i$, $\Lambda_i$,
$\sigma_i$ and $\tau_i$ satisfy all the properties stated in the above points~1 to 4 for
the case when $i = r$, which means that this step terminates the recursion.

Otherwise, if a path $\alpha \subset \Delta_{i-1}$ exists joining different terminal edges
of $\Delta_{i-1}$ ending at vertices in $U_{i-1}$, then we can choose it to be minimal
(hence simple) and apply Lemma \ref{step-p/thm} to the restriction of $\phi_{i-1}$ to
$\Delta_{i-1}$ and to such a minimal $\alpha$, in order to get the structure described in
that statement for $\cup_j\, \phi_{i-1}^j(\alpha) \subset \Delta_{i-1}$ up to
$F$-equivalence. Such an $F$-equivalence only involves internal edges of $\Delta_{i-1}$,
hence it does not change the set of free ends $U_{i-1} \subset \Delta_{i-1}$ and the
restriction of $\phi_{i-1|\Delta_{i-1}}$. Therefore, it can be extended to an
$F$-equivalence of the whole $\phi_{i-1}$ on $\Gamma_{i-1}$, which leaves $\Lambda_{i-1}$ 
and the restriction $\phi_{i-1|\Lambda_{i-1}}$ unchanged.

As a result we get a new uni/trivalent graph $\Gamma_i$, such that $\Lambda_{i-1} \subset
\Gamma_i$ and a new automorphism $\phi_i: \Gamma_i \to \Gamma_i$, which is $F$-equivalent
to $\phi_{i-1}$ and coincides with $\phi_{i-1}$ on $\Lambda_{i-1}$. We also get a new
minimal path $\alpha \subset \Cl(\Gamma_i - \Lambda_{i-1})$, which join different
$\phi_i$-equivalent terminal edges of $\Cl(\Gamma_i - \Lambda_{i-1})$ ending at vertices
in $U_{i-1}$, such that $\cup_j\, \phi_i^j(\alpha)$ itself (no more up to $F$-equivalence)
has the structure described in Lem\-ma~\ref{step-p/thm}. Let $e \subset \Gamma_i$ be the
terminal edge of $\Cl(\Gamma_i - \Lambda_{i-1})$ ending at $u_{i-1}$, and $v \in
\Gamma_i$ be the other end of $e$, which is a trivalent vertex of $\Cl(\Gamma_i -
\Lambda_{i-1})$. Up to $\phi_i$-equiv\-alence, we can assume that $\alpha \subset
\Gamma_i$ starts at $v$ and ends at $\phi_i^s(v)$ with $0 < s < n$. According to Lemma
\ref{step-p/thm}, we can also assume that $\alpha$ consists of either one edge $a$ or two
edges $a$ and $b$ sharing the vertex $u$, while $\cup_j\, \phi_i^j(\alpha)$ consists of
the disjoint union $C \sqcup \phi_i(C) \sqcup \dots \sqcup \phi_i^{n/\ell - 1}(C)$ of
$n/\ell$ cycles (possibly a single one, for $\ell = n$), where $C = \alpha
\,\phi_i^s(\alpha) \cdots \phi_i^{(\ell - 1) s}(\alpha)$ for some $\ell > 1$ such that $s
\ell = t n$ with $(t,\ell) = 1$.

We put $\Lambda_i = \Lambda_{i-1} \cup_j\, \phi_i^j(e) \cup_j\, \phi_i^j(\alpha)$ and
define $\sigma_i$ and $\tau_i$ to be the unique automorphisms of $\Lambda_i$ extending
$\sigma_{i-1}$ and $\tau_{i-1}$ respectively. To see that such extensions exist, we first
define $\sigma_i(\phi_i^j)(v) = \phi_i^{-j}(v)$ and $\tau_i(\phi_i^j)(v) =
\phi_i^{1-j}(v)$, and also $\sigma_i(\phi_i^j)(u) = \phi_i^{-j}(u)$ and
$\tau_i(\phi_i^j)(u) = \phi_i^{1-j}(u)$ in the case when $\alpha = ab$, for every $j =
0, \dots, n-1$. Then, essentially by the same argument used above for $i = 1$, we check
that these definitions, together with $\sigma_i(\phi_i^j)(u_{i-1}) = \phi_i^{-j}(u_{i-1})$
and $\tau_i(\phi_i^j)(u_{i-1}) = \phi_i^{1-j}(u_{i-1})$ for every $i=0, \dots, n$,
preserve adjacency of vertices in $\Lambda_i$, hence they determine automorphisms
$\sigma_i$ and $\tau_i$ of the graph $\Lambda_i$ extending $\sigma_{i-1}$ and $\tau_{i-1}$
respectively. Notice that $\tau_i \circ \sigma_i$ trivially coincides with
$\phi_{i|\Lambda_i}$. Moreover, $\sigma_i$ fixes the edge $e$ and reverses the edge
$\phi_i^{(n+1)/2}(a)$, while $\tau_i$ fixes the edge $\phi_i^{(n+1)/2}(e)$ and reverses
the edge $a$. Finally, $u_i = \phi_i^{(n+1)/2}(u)$ is the unique vertex in its
$\phi_i$-orbit $U_i = \{\phi_i^j(u)\,|\,j = 0, \dots, n\}$ fixed by $\sigma_i$.

At this point, if $\alpha$ consists of the only edge $a$, hence $C$ has the structure
depicted in the left side of Figure \ref{onion/fig}, then $\Lambda_i = \Gamma_i$ by the
connectedness of $\Gamma_i$ and the recursion terminates with $r = i$. Otherwise, if
$\alpha$ consists of the two edges $a$ and $b$, hence $C$ has the structure depicted in
the right side of Figure \ref{onion/fig}, then we put $\Delta_i = \Cl(\Gamma_i -
\Lambda_i)$ and observe that this is a uni/trivalent graph such that $\Delta_i \cap
\Lambda_i = U_i$. This conclude the recursive step.
\end{proof}


\section{The case of order $3^m$%
\label{order-3m/sec}}

In this section, we want to prove that any automorphism $\phi \in \A_{g,b}$ with order 
$\ord(\phi) = 3^m$ is $F$-equivalent to the composition of two automorphisms of order $2$.

Like in the prevous section, a weaker assumption will suffice to get the same conclusion.
Namely, we will assume that $n = \ord(\phi)$ is an odd multiple of $3$. The only 
properties we will need, for any automorphism $\phi: \Gamma \to \Gamma$ of a (possibly 
discon\-nected) uni/trivalent graph $\Gamma$ with such an order $n$, are the following:
\begin{itemize}
\item[1)] $\ord_\phi(e) = n$ for any edge $e$ of $\Gamma$, while $\ord_\phi(v)$ is either
$n$ or $n/3$ for any vertex $v$ of $\Gamma$, thanks to Lemma~\ref{orders/thm};
\item[2)] the situation of Lemmas~\ref{paths-diagonal/thm} cannot occur, while that of
Lemma \ref{paths-doubled/thm} can only occur with the two paths $\phi^i(\alpha)$ and
$\phi^j(\alpha)$ sharing no edge.
\end{itemize}

The whole argument is essentially the same as in the previous section, except for some 
more cases occurring here. Hence, in both statements and proofs we will just concern with 
those extra cases, while referring for the others to the analogous statements and proofs 
in the previous section.

\begin{lemma}\label{cycle-3/thm}
Let $\phi: \Gamma \to \Gamma$ be an automorphism of a (possibly disconnected)
uni/trivalent graph $\Gamma$, whose order $n = \ord(\phi)$ is an odd multiple of $3$. If
$\alpha \subset \Gamma$ is a (simple) path of minimal length among all paths joining any
two distinct $\phi$-equivalent vertices, then $\cup_i\, \phi^i(\alpha)$ is one
of the following:
\begin{itemize}
\item[1)] a disjoint union $C \sqcup \phi(C) \sqcup \dots \sqcup \phi^{n/\ell - 1}(C)$
of $n/\ell$ simple cycles (possibly a single one, for $\ell = n$), each given by a
concatenation of images of $\alpha$, having the same structure as described in Lemma 
\ref{cycle-p/thm}; in this case, up to $F$-equivalence we can assume $\alpha$ to 
consist of a single edge $a$ (see Figure \ref{cycle/fig}, where $a_i$ stands for 
$\phi^{is}(a)$ and a similar notation is adopted for the vertices $v_i$ as well);
\item[2)] a disjoint union $T \sqcup \phi(T) \sqcup \dots \sqcup \phi^{n/3\,-\,1}(T)$ of
$n/3$ tripods (possibly a single one, for $n=3$), where $T = \alpha \cup
\phi^{n/3}(\alpha) \cup \phi^{2n/3}(\alpha)$ and $\alpha$ is the concatenation of two
edges $a$ and $b$, such that $b = \phi^{n/3}(\bar a)$, $\phi^{n/3}(b) = \phi^{2n/3}(\bar
a)$ and $\phi^{2n/3}(b) = \bar a$ (see Figure~\ref{tripode1/fig}, where $a_i$ stands for
$\phi^{i \mkern1mu n/3}(a)$ and a similar notation is adopted for the vertices $v_i$, 
which can be either all univalent as on the left side or all trivalent as on the right 
side).
\end{itemize}
\end{lemma}

\begin{Figure}[htb]{tripode1/fig}
\vskip-21pt
\fig{}{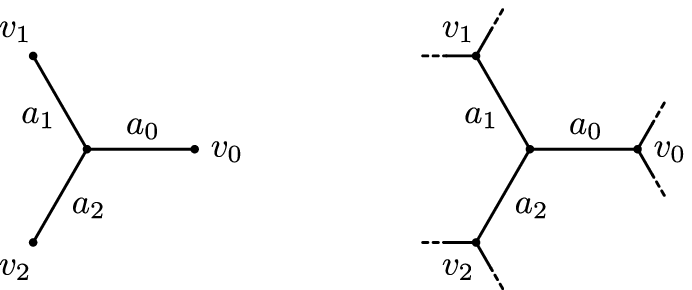}[-3pt]
{}{The form of the tripode $T$ in Lemma \ref{cycle-3/thm}.}
\end{Figure}

\begin{proof}
Let $\alpha \subset \Gamma$ be an arbitrary path of minimal length among all paths joining 
any two distinct $\phi$-equivalent vertices and let $v \neq v'$ be its ends. For any $i
\neq j \!\!\mod n$, let us consider the two paths $\phi^i(\alpha)$ and $\phi^j(\alpha)$,
and the four different situations described in Lemmas~\ref{paths-disjoint/thm} to
\ref{paths-doubled/thm}, which cover all the possibilities, thanks to the
$\phi$-equivalence of $v$ and $v'$.

The situation of Lemma~\ref{paths-diagonal/thm} cannot occur, while that of Lemma
\ref{paths-doubled/thm} could possibly occur only with $\phi^i(\alpha)$ and
$\phi^j(\alpha)$ sharing no edge, but the argument below shows this is actually
impossible.

Assume by contradiction that $\phi^i(\alpha)$ and $\phi^j(\alpha)$ are as in Lemma
\ref{paths-doubled/thm}, hence they meet at $\phi^i(v) = \phi^j(v)$ and $\phi^i(v') =
\phi^j(v')$. Let us denote by $a$ the first edge of $\alpha$ starting from $v$ and by $u$
the other end of $a$. Then, the vertices $\phi^i(u)$ and $\phi^j(u)$ are distinct since
$\phi^i(\alpha)$ and $\phi^j(\alpha)$ do not share any edge. Being $n$ an odd multiple of
3, thanks to Lemma~\ref{orders/thm} we have $\ord_\phi(v) = n/3$ and $\ord_\phi(u)= n$.

As a consequence, we get $\len(\alpha) = 2$. In fact, the concatenation $\phi^i(\bar a)\,
\phi^j(a)$ is simple path of length $2$ between the two distinct $\phi$-equivalent
vertices $\phi^i(u)$ and $\phi^j(u)$. Hence, by the global minimality of $\alpha$ we have
$\len(\alpha) \leq 2$. On the other hand, $\alpha$ cannot be reduced to the single edge
$a$, otherwise $v$ and $u$ should be $\phi$-equivalent, in contrast with $\ord_\phi(u)
\neq \ord_\phi(v)$.

Now, let $s$ be an integer such that $v' = \phi^s(v)$. Then, $\phi^s(u)$ coincides with 
one of $u$, $\phi^{n/3}(u)$ and $\phi^{2n/3}(u)$, since these are the three vertices
adjacent to $v'$. In any case, we can conclude that $s$ is a multiple of $n/3$, which is 
in contrast with $\ord_\phi(v) = n/3$ and $v' \neq v$. This proves that the situation of 
Lemma \ref{paths-doubled/thm} cannot occur.

At this point, we have that $\phi^i(\alpha)$ and $\phi^j(\alpha)$ can only meet as in
Lemma~\ref{paths-adjacent/thm}. Let $\phi^i(u)$ with $u \in \alpha$ be the common end
of $\delta_{i,1}$ and $\delta_{j,k}$. As in the proof of Lemma~\ref{cycle-p/thm}, we can 
see that either $\phi^j(u) = \phi^j(v')$ or $\phi^j(u) = \phi^i(u)$, due to the global 
minimality of $\alpha$. But differently from that proof, here the latter equality can 
actually hold when $u$ has order $n/3$ and $j = i \!\!\mod n/3$.

If $\phi^j(u) = \phi^i(u)$, then the global minimality of $\alpha$ implies that both
$\delta_{i,1}$ and $\delta_{j,k}$ have length 1. Therefore, $\alpha$ is the concatenation
of two edges $a$ and $b$ sharing the vertex $u$, and hence $T = \alpha \cup
\phi^{n/3}(\alpha) \cup \phi^{2n/3}(\alpha)$ has the form described in point~2 of
the statement. Of course, this means that $\cup_i\,\phi^i(\alpha) = T \cup \phi(T) \cup
\dots \cup \phi^{n/3\,-\,1}(T)$. Moreover, any two different copies of $T$ in this union 
are disjoint, otherwise they should share a free end, and there would be two paths
$\phi$-equivalent to $\alpha$ exiting from that common vertex like in 
Lemma~\ref{paths-doubled/thm}.

Notice that the last argument shows that, if $\phi^j(u) = \phi^i(u)$ for some $i$ and $j$,
then the same holds for any $i$ and $j$ such that $\phi^i(\alpha)$ and $\phi^j(\alpha)$
are not disjoint.

So, we are left with the case when $\phi^j(u) = \phi^j(v')$ for any two non-disjoint paths
$\phi^i(\alpha)$ and $\phi^j(\alpha)$. In this case, the same argument exploited in the
proof of Lemma~\ref{cycle-p/thm}, allows us to conclude that the situation is the one
described in point~1 of the statement.
\end{proof}

\begin{lemma}\label{step-3/thm}
Let $\phi: \Gamma \to \Gamma$ be an automorphism of a (possibly disconnected)
uni/trivalent graph $\Gamma$, whose order $n = \ord(\phi)$ is an odd multiple of $3$. If
$\alpha \subset \Gamma$\break is a (simple) path, possibly degenerate to a single vertex,
of minimal length among all paths joining any two distinct terminal edges in a given
$\phi$-orbit, then up to $F$-equivalence we can assume $\cup_i\,\phi^i(\alpha)$ to
be one of the following:
\begin{itemize}
\item[1)] a disjoint union $C \sqcup \phi(C) \sqcup \dots \sqcup \phi^{n/\ell - 1}(C)$
of $n/\ell$ simple cycles (possibly a single one, for $\ell = n$), each given by a
concatenation of images of $\alpha$, having the same structure as described in
Lemma~\ref{step-p/thm}, with $\alpha$ consisting of either one edge $a$ or two edges $a$
and $b$ (see Figure~\ref{onion/fig}, where $a_i$ and $b_i$ stand for $\phi^{is}(a)$ and
$\phi^{is}(b)$ respectively, and a similar notation is adopted for the terminal edges
$e_i$'s and the vertices $v_i$'s as well);
\item[2)] a disjoint union $T \sqcup \phi(T) \sqcup \dots \sqcup \phi^{n/3\,-\,1}(T)$ of
$n/3$ tripods (possibly a single one, for $n=3$), where $T = \alpha \cup
\phi^{n/3}(\alpha) \cup \phi^{2n/3}(\alpha)$, with $\alpha$ being the concatenation of two
edges $a$ and $b$, such that $b = \phi^{n/3}(\bar a)$, $\phi^{n/3}(b) = \phi^{2n/3}(\bar
a)$ and $\phi^{2n/3}(b) = \bar a$ (see left side of Figure~\ref{tripode2/fig}, where $a_i$
and $e_i$ stand for $\phi^{i \mkern1mu n/3}(a)$ and $\phi^{i \mkern1mu n/3}(e)$
respec\-tively);
\item[3)] a set of $n/3$ trivalent vertices, when $\alpha$ reduces to a single vertex,
that is the terminal edges it joins share a common trivalent end; in this case the whole
graph $\Gamma$ consists of a disjoint union $\hat T \sqcup \phi(\hat T) \sqcup \dots
\sqcup \phi^{n/3\,-\,1}(\hat T)$ of $n/3$ tripods, where $\hat T = e \cup \phi^{n/3}(e)
\cup \phi^{2n/3}(e)$ with $\alpha$ reduced to the trivalent vertex of $\hat T$\break (see 
right side of Figure~\ref{tripode2/fig}, where $e_i$ stands for the terminal edge $\phi^{i
\mkern1mu n/3}(e)$).
\end{itemize}
\end{lemma}

\begin{Figure}[htb]{tripode2/fig}
\vskip-18pt
\fig{}{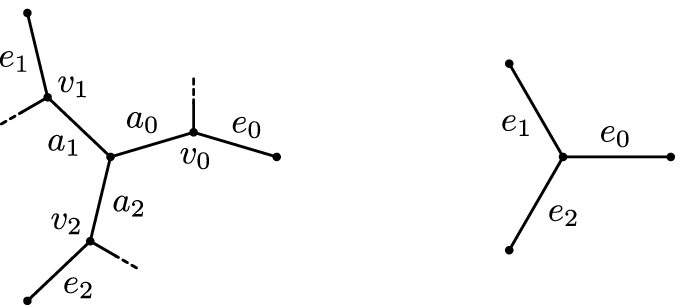}
{}{The form of the tripodes $T$ (left) and $\hat T$ (right) in Lemma \ref{step-3/thm}.}
\end{Figure}

\begin{proof}
Let $\alpha$ be a path as in the statement and let $e \neq e'$ be the terminal edges it
joins. Then the ends of $\alpha$ coincide with the unique trivalent ends $v$ and $v'$ 
of $e$ and $e'$ respectively.

If $v = v'$, that is $\len(\alpha) = 0$, then $\ord_\phi(v) = n/3$ according to
Lemma~\ref{orders/thm} and $e'$ is either $\phi^{n/3}(e)$ or $\phi^{2n/3}(e)$. In any
case, $\phi^{n/3}$ cyclically permutes the three edges of the tripode $\hat T = e \cup
\phi^{n/3}(e) \cup \phi^{2n/3}(e)$, and we have the situation described in point~3 of the
statement.

If $v \neq v'$, that is $\len(\alpha) \geq 1$, we consider any two distinct paths
$\phi^i(\alpha)$ and $\phi^j(\alpha)$ with $i \neq j \!\!\mod n$ and look once again at
the four possible situations described in Lemmas~\ref{paths-disjoint/thm} to
\ref{paths-doubled/thm}. 

As in the proof of Lemma~\ref{cycle-3/thm}, the situations of
Lemmas~\ref{paths-diagonal/thm} and \ref{paths-doubled/thm} cannot occur. But here the
argument to exclude the latter is different. Namely, if $\phi^i(\alpha)$ and
$\phi^j(\alpha)$ were as in Lemma~\ref{paths-doubled/thm}, then $\phi^{j-i}$ should swap
the first edges of them starting from $\phi^i(v) = \phi^j(v)$, being $\phi^i(e) =
\phi^j(e)$ the third edge at that vertex, and this would be in contrast with the oddness
of $n$.

Therefore, $\phi^i(\alpha)$ and $\phi^j(\alpha)$ can meet only as in
Lemma~\ref{paths-adjacent/thm}, and their common end $\phi^j(v) = \phi^i(v')$ cannot be
shared by any other $\phi^h(\alpha)$ (cf. proof of Lemma~\ref{cycle-p/thm}). Then, 
consider the two subpaths $\delta_{i,1} \subset \phi^i(\alpha)$ and $\delta_{j,k} \subset
\phi^j(\alpha)$ in Figure~\ref{paths-adjacent/fig}.

\begin{Figure}[b]{tripode3/fig}
\vskip3pt
\fig{}{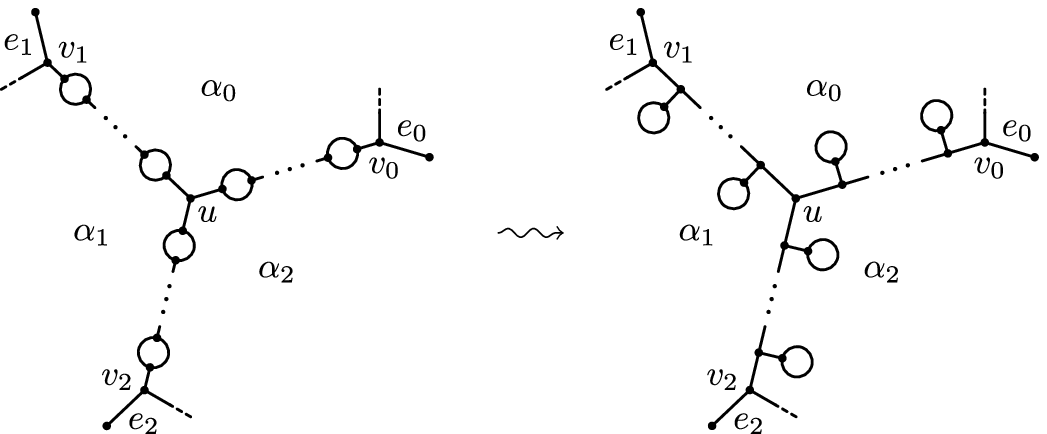}
{}{The starting configuration $T$ in the proof of Lemma \ref{cycle-3/thm}.}
\end{Figure}

If $\len(\delta_{i,1}) = \len(\delta_{j,k}) = \len(\alpha)/2$, then for the middle vertex
$u$ of $\alpha$ we have that $\phi^i(u) = \phi^j(u)$ is the common end of $\delta_{i,1}$
and $\delta_{j,k}$. By the same argument as above, this implies that $\ord_\phi(u) = n/3$ 
and $\phi^{n/3}$ cyclically permutes the three paths $\alpha$, $\phi^{n/3}(\alpha)$ and 
$\phi^{2n/3}(\alpha)$. Moreover, such three paths are disjoint from any other path 
$\phi^h(\alpha)$, hence we can conlcude that $\cup_i\,\phi^i(\alpha)$ is a disjoint
union $T \sqcup \phi(T) \sqcup \dots \sqcup \phi^{n/3\,-\,1}(T)$, with $T$ as in the left 
side of Figure~\ref{tripode3/fig}. Here, $\alpha_i$ stands for the path $\phi^{i \mkern1mu 
n/3}(\alpha)$, which is the part of $T$ in the corresponding circular sector, and any arc 
except the terminal ones represent a path of edges. Analogously, $v_i$ and $e_i$ stand for 
$\phi^{i \mkern1mu n/3}(v)$ and $\phi^{i \mkern1mu n/3}(e)$ respectively.

We can reduce any arc in this starting configuration $T$ to a single edge,
by performing $\phi$-invariant slidings as in the proof of Lemma~\ref{step-p/thm} (cf.
Figure~\ref{onion-pf2/fig}), in order to push all the edges attached along the
$\alpha_i^s$ so that that they becomes attached to the last one exiting from $v_i$ (and
different from $e_i$). Then, the obvious $\phi$-invariant $F$-moves allow us to 
change the resulting $T$ into the form depicted on the right side of 
Figure~\ref{tripode3/fig}, with each $\alpha_i$ being now the shortest path between $v_i$ 
and $v_{i+1 \!\!\mod 3}$. Finally, we can get the configuration described in point~2 of 
the statement and shown in the left side of Figure~\ref{tripode2/fig}, by further 
$\phi$-invar\-iant slidings of the new edges attached along the actual $\alpha_i$ as 
above.

Otherwise, if $\len(\delta_{i,1}) = \len(\delta_{j,k}) > \len(\alpha)/2$, then the same 
argument as in the proof of Lemma~\ref{step-p/thm} still works here, to put 
$\cup_i\,\phi^i(\alpha)$ in the form stated in point~1 of the statement.
\end{proof}

\begin{proposition}\label{onion-3/thm}
Let $\phi \in \A_{g,b}$ be a non-trivial automorphism, whose order $n = \ord(\phi) > 1$ is
an odd multiple of $3$. Then, $\phi$ is $F$-equivalent to the composition $\tau \circ
\sigma$ of two automorphisms $\sigma, \tau \in \A_{g,b}$ such that $\ord(\sigma) =
\ord(\tau) = 2$. Moreover, both $\sigma$ and $\tau$ can be assumed to fix an edge and
reverse an invariant edge.
\end{proposition}

\begin{proof}
Given $\phi: \Gamma \to \Gamma$ as in the statement, the same recursive construction of
the proof of Proposition~\ref{onion-p/thm} will provide a sequence of automorphisms
$\phi_i: \Gamma_i \to \Gamma_i$ and a sequence of subgraphs $\Lambda_i \subset \Gamma_i$
with $i = 0, \dots, r$, satisfying the properties required in that proof, except for the
following facts: 
\begin{itemize}
\item[{\sl a\/})] $\Lambda_1$ is allowed to be a subgraph of the first
barycentric subdivision of $\Gamma_1$ instead that a subgraph of $\Gamma_1$ itself;
\item[{\sl b\/})] $U_i$ is required to have cardinality $n$ for $i = 1, \dots, r-1$
(Lemma~\ref{orders/thm} ensures that such cardinality is either $n$ or $n/3$). 
\end{itemize}
In particular, all the $\phi_i$ are $F$-equivalent to $\phi$, and all the restrictions
$\phi_{i|\Lambda_i}$ admit a factorization $\tau_i \circ \sigma_i$ into two involutions of
$\Lambda_i$. Hence, $\phi$ turns out to be $F$-equivalent to $\phi_r = \tau_r \circ
\sigma_r$ (being $\Lambda_r = \Gamma_r$), which proves the proposition.

The starting step of the recursion is provided by Lemma~\ref{cycle-3/thm}. If the
situation described in point 1 of that lemma occurs, we define the automorphism $\phi_1:
\Gamma_1 \to \Gamma_1$, the subgraph $\Lambda_1 \subset \Gamma_1$, the two involutions
$\sigma_1,\tau_1: \Lambda_1 \to \Lambda_1$ and the vertex $u_1$, as in the proof of
Proposition~\ref{onion-p/thm}. Otherwise, the situation described in point 2 of
Lemma~\ref{cycle-3/thm} occurs. In this case, there exists a $\phi$-invariant disjoint
union of $n/3$ tripodes $T \sqcup \phi(T) \sqcup \dots \sqcup \phi^{n/3\,-\,1}(T) \subset
\Gamma$, with $\phi^{n/3}$ cyclically permuting the edges of the tripode $T$. Then, we put
$\Gamma_1 = \Gamma_0 = \Gamma$, $\phi_1 = \phi_0 = \phi$ and $\Lambda_1 = T' \sqcup
\phi(T') \sqcup \dots \sqcup \phi^{n/3\,-\,1}(T') \subset \Gamma_1$, where $T' \subset T$
is the star of the trivalent vertex of $T$ in the first barycentric subdivision. Moreover, 
we denote by $v$ the trivalent vertex of $T'$ and by $u$ any free end of $T'$, and 
we define $\sigma_1$ and $\tau_1$ by putting $\sigma_1(\phi^j(v)) = \phi^{-j}(v)$, 
$\sigma_1(\phi^j(u)) = \phi^{-j}(u)$, $\tau_1(\phi^j(v)) = \phi^{1-j}(v)$ and
$\tau_1(\phi^j(u)) = \phi^{1-j}(u)$ for $j = 0, \dots, n-1$. 

In order to conclude the starting step, it is enough to observe that: $\sigma_1$ and
$\tau_1$ are well-defined involutive automophisms of $\Lambda_1$, since they preserve the
adjacency of vertices, being $\phi^j(v)$ adjacent to $\phi^k(u)$ if and only if $j = k
\!\!\mod n/3$; $\tau_1 \circ \sigma_1$ coincides with the restriction
$\phi_{1|\Lambda_1}$; $\Delta_1 = \Cl(\Gamma_1 - \Lambda_1)$ is a uni/trivalent graph;
$U_1 = \Delta_1 \cap \Lambda_1$ consists of the $\phi_1$-orbit of the univalent vertex $u$
of $\Delta_1$ and it has cardinality $n$; $u_1 = u$ is the only vertex of $U_1$ fixed by
$\sigma_1$.

Now, to realize the recursive step of the construction, assume we are given $\phi_{i-1}:
\Gamma_{i-1} \to \Gamma_{i-1}$, $u_{i-1} \in U_{i-1} \subset \Lambda_{i-1} \varsubsetneq
\Gamma_{i-1}$ and $\sigma_{i-1},\tau_{i-1}: \Lambda_{i-1} \to \Lambda_{i-1}$ with the
properties required in the proof of Proposition~\ref{onion-p/thm} for $i-1 < r$,
integrated by the points {\sl a\/} and {\sl b\/} said above. In particular, the
$\phi_{i-1}$-orbit $U_{i-1}$ has cardinality $n$. Therefore, since $\ord(\phi_{i-1}) = n$
by Lemma~\ref{ordinv/thm}, also the restriction of $\phi_{i-1}$ to the uni/tri\-valent
graph $\Delta_{i-1} = \Cl(\Gamma_{i-1} - \Lambda_{i-1})$ is an automorphism of the same
order $n$. Furthermore, the terminal edges of $\Delta_{i-1}$ ending at vertices in
$U_{i-1}$ are all distinct (see proof of Proposition~\ref{onion-p/thm}).

If $\Delta_{i-1}$ does not contain any path joining different terminal edges ending at
vertices in $U_{i-1}$, then we can terminate the recursion with the $i$-th step, in the 
same way as in the proof of Proposition~\ref{onion-p/thm}.

Otherwise, if a path $\alpha \subset \Delta_{i-1}$ exists joining different terminal edges
of $\Delta_{i-1}$ ending at vertices in $U_{i-1}$, then we can choose it to be minimal
(hence simple) and apply Lemma \ref{step-3/thm} to the restriction of $\phi_{i-1}$ to
$\Delta_{i-1}$ and to such a minimal $\alpha$, in order to get the structure described in
that statement for $\cup_j\, \phi_{i-1}^j(\alpha) \subset \Delta_{i-1}$ up to
$F$-equivalence. Such an $F$-equivalence only involves internal edges of $\Delta_{i-1}$,
hence it does not change the set of free ends $U_{i-1} \subset \Delta_{i-1}$ and the
restriction of $\phi_{i-1|\Delta_{i-1}}$. Therefore, it can be extended to an
$F$-equivalence of the whole $\phi_{i-1}$ on $\Gamma_{i-1}$, which leaves $\Lambda_{i-1}$ 
and the restriction $\phi_{i-1|\Lambda_{i-1}}$ unchanged.

As a result, we get a new uni/trivalent graph $\Gamma_i$, such that $\Lambda_{i-1} \subset
\Gamma_i$ and a new automorphism $\phi_i: \Gamma_i \to \Gamma_i$, which is $F$-equivalent
to $\phi_{i-1}$ and coincides with $\phi_{i-1}$ on $\Lambda_{i-1}$. We also get a new
minimal path $\alpha \subset \Cl(\Gamma_i - \Lambda_{i-1})$, which joins different
$\phi_i$-equivalent terminal edges of $\Cl(\Gamma_i - \Lambda_{i-1})$ ending at vertices
in $U_{i-1}$, such that $\cup_j\, \phi_i^j(\alpha)$ itself (no more up to
$F$-equivalence) has the structure stated in Lemma~\ref{step-3/thm}. Then, we proceed by 
distinguishing the different situations described in points 1 to 3 of that lemma.

If the situation of point~1 occurs, then the $i$-th step is the same as in the proof of
Proposition~\ref{onion-p/thm}. In particular, if such step is not the concluding one, then
$U_i$ can be easily seen to have cardinality $n$.

In the case when the situation of point~2 occurs, we denote by $e \subset \Gamma_i$ the
terminal edge of $\Cl(\Gamma_i - \Lambda_{i-1})$ whose free end is $u_{i-1}$. Then, we
choose $u_i$ to be the other end of $e$ and put $\Lambda_i = \Lambda_{i-1} \cup_j
\phi_i^j(e) \cup_j \phi_i^j(T)$, in such a way that $\Delta_i = \Gamma_i - \Lambda_i$ is a
uni/trivalent graph, which meets $\Lambda_i$ at the $\phi_i$-orbit $U_i$ of $u_i$.
Moreover, we define $\sigma_i$ and $\tau_i$ to be the unique automorphisms of $\Lambda_i$
extending $\sigma_{i-1}$ and $\tau_{i-1}$ respectively. The existence of such extensions
and all the required properties of them, in particular the equality $\tau_i \circ \sigma_i
= \phi_{i|\Lambda_i}$ and the fact that $u_i$ is the unique vertex of $U_i$ fixed by
$\sigma_i$, can be proved by the same argument as in the proof of Lemma~\ref{onion-p/thm}.

Finally, if the situation of point~3 occurs, we just put $\Lambda_i = \Lambda_{i-1}
\cup_j\, \phi^j_i(\hat T) = \Gamma_i$, where the last equality is due to the connectedness
of $\Gamma_i$, and define $\sigma_i$ and $\tau_i$ to be the unique extensions of
$\sigma_{i-1}$ and $\tau_{i_1}$ to $\Gamma_i$. These can be easily seen to verify all the
required properties for $i = r$, hence the recursion terminates with this step.
\end{proof}


\section{The case of order $2^m$%
\label{order-2m/sec}}

In order to deal with automorphisms $\phi: \Gamma \to \Gamma$ in $\A_{g,b}$ of order
$2^m$, we first prove that we can limit ourselves to consider the case when $\ord_\phi(e)
= 2^m$ for every edge $e$ of $\Gamma$ (hence, according to Lemma \ref{orders/thm},
$\ord_\phi(v) = 2^m$ for every vertex $v$ of $\Gamma$).

\begin{lemma}\label{reduction/thm}
Let $\phi:\Gamma \to \Gamma$ be an automorphism in $\A_{g,b}$ with $\ord(\phi) = 2^m$, and
let $2^k = \min_{e \in \Gamma} \ord_\phi(e)$, where $e$ varies among all the edges of 
$\Gamma$. Then, up to $F$-equivalence and composition with elementary automorphisms, 
$\phi$ can be reduced to an automorphism $\psi: \Gamma' \to \Gamma'$ such that $\ord(\psi) 
= 2^k$ and $\ord_\psi(e) = 2^k$ for every edge $e$ of $\Gamma'$.
\end{lemma}

\begin{proof}
The proof is by induction on the pair $(m,n_m)$ with respect to the lexicographic order,
where $n_m$ is the number of the edges $e$ of $\Gamma$ such that $\ord_\phi(e) = 2^m$.

The base of the induction is for $m = k$, and hence $\ord_\phi(e) = 2^m$ for all the edges  
$e$ of $\Gamma$, that is $n_m = 3g - 3 + 2b$. 

If on the contrary $m > k$, then Lemma \ref{orders/thm} tells us that there 
are three distinct edges $e_0,e_1$ and $e_2$ of $\Gamma$ sharing a trivalent vertex $v$ of 
$\Gamma$, such that $\ord_\phi(e_0) = \ord_\phi(v) = 2^{m-1}$, $\ord_\phi(e_1) = 
\ord_\phi(e_2) = 2^m$ and $e_2 = \phi^{2^{m - 1}}(e_1)$ (note that $e_1 \neq e_2$, 
otherwise $e_1 = e_2$ would be a loop and its order should be $2^{m-1}$). Let $v_1$ and 
$v_2$ the other ends of $e_1$ and $e_2$ respectively.

If $v_1$ and $v_2$ are (distinct) univalent vertices, then by composing with the terminal
switch $S_{e_1,e_2}$ we get an automorphism $\rho = S_{e_1,e_2} \circ \phi$ with either
$m(\rho) < m(\phi)$ or $m(\rho) = m(\phi)$ and $n_m(\rho) < n_m(\phi)$. In fact,
$\ord_\rho(e) = \ord_\phi(e)/2$ for all the edges $e$ in the $\phi$-orbit of $e_1$ (and
$e_2$), while $\ord_\rho(e) = \ord_\phi(e)$ for all the other edges of $\Gamma$.

If $v_1$ and $v_2$ are coinciding trivalent vertices, then the same argument as above 
applies, but with $S_{e_1,e_2}$ an internal switch.

Finally, assume that $v_1$ and $v_2$ are distinct trivalent vertices. Then, we have the 
situation depicted in the left side of Figure \ref{reduction/fig}, where 
$\phi^{2^{m-1}}(a_i) = a_{3-i}$ and $\phi^{2^{m-1}}(b_i) = b_{3-i}$ for $i = 1,2$ (with 
the four edges $a_1,a_2, b_1$ and $b_2$ not necessarily distinct). 
\begin{Figure}[htb]{reduction/fig}
\fig{}{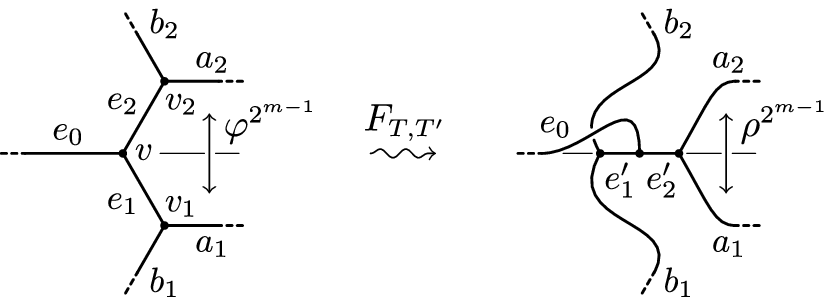}
{}{Reducing $(m(\phi),n_m(\phi))$.}
\end{Figure}
In this case, we can perform the $\phi$-invariant move $F_{\cal T,\cal T'}$ acting on a
neighborhood of $e_1 \cup e_2$ as described in Figure \ref{reduction/fig}. This changes
$\phi$ into an automorphism $\rho$ with either $m(\rho) < m(\phi)$ or $m(\rho) = m(\phi)$
and $n_m(\rho) < n_m(\phi)$. In fact $F_{\cal T,\cal T'}$, replaces the $\phi$-orbit of
$e_1$ (and $e_2$) by two different $\rho$-orbits of $e'_1$ and $e'_2$ respectively, having
cardinality $2^{m-1}$. This completes proof of the inductive step.
\end{proof}

In the light of the previous lemma, the next two lemmas concern the special case of an 
automorphism $\phi: \Gamma \to \Gamma$ with $\ord(\phi) = n = 2^m$ and $\ord_\phi(e) = n$ 
for every edge $e$ of $\Gamma$. We notice that, under the latter assumption, the situation 
described in Lemma \ref{paths-doubled/thm} can occur only with $\phi^i(\alpha) = 
\phi^j(\alpha)$, that is $i = j \!\!\mod n$.

\pagebreak

\begin{lemma}\label{cycle-2/thm}
Let $\phi: \Gamma \to \Gamma$ be an automorphism of a (possibly disconnected)
uni/trivalent graph $\Gamma$, with $\ord(\phi) = n = 2^m$ and $\ord_\phi(e) = n$ for every
edge $e$ of $\,\Gamma$.\break If $\alpha \subset \Gamma$ is a (simple) path of minimal
length among all paths joining any two distinct $\phi$-equivalent vertices, then $\cup_i\,
\phi^i(\alpha)$ is one of the following:
\begin{itemize}
\item[1)] a disjoint union $C \sqcup \phi(C) \sqcup \dots \sqcup \phi^{n/\ell - 1}(C)$
of $n/\ell$ simple cycles (possibly a single one, for $\ell = n$), each given by a
concatenation of images of $\alpha$, having the same structure as described in Lemma
\ref{cycle-p/thm}; in this case, up to $F$-equivalence we can assume $\alpha$ to consist
of a single edge $a$ (see Figure \ref{cycle-2/fig}, where $\ell$ is a power of $2$, 
$a_i$ stands for $\phi^{is}(a)$ and a similar notation is adopted for the vertices $v_i$ 
as well);

\begin{Figure}[htb]{cycle-2/fig}
\fig{}{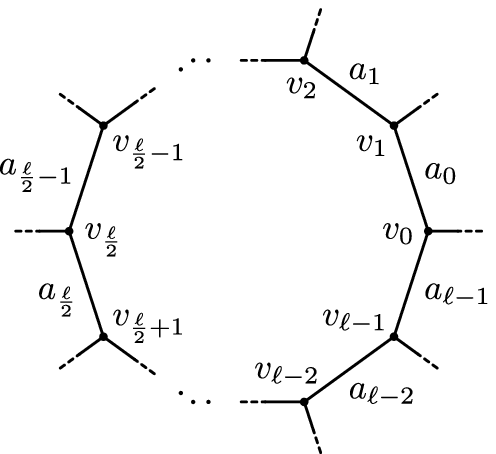}
{}{The form of the cycle $C$ in Lemma~\ref{cycle-2/thm}.}
\end{Figure}

\item[2)] a disjoint union $D \sqcup \phi(D) \sqcup \dots \sqcup \phi^{n/2\,-\,1}(D)$ of
$n/2$ ``diagonal'' edges (possibly a single one, for $n = 2$), where $D = \alpha$ and
$\phi^{n/2}(D) = \bar D$ (see Figure~\ref{diagonal1/fig}, where the vertices $v_0$ and
$v_1$ are switched by $\phi^{n/2}$, and they can be either both univalent as on the left
side or both trivalent as on the right side).
\end{itemize}
\end{lemma}

\begin{Figure}[htb]{diagonal1/fig}
\vskip-6pt
\fig{}{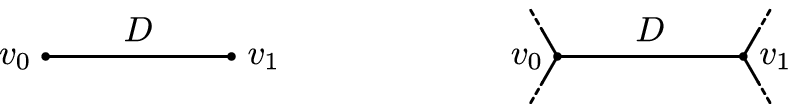}
{}{The ``diagonal'' edge $D$ in Lemma~\ref{cycle-2/thm}.}
\end{Figure}

\begin{proof}
Let $\alpha \subset \Gamma$ be an arbitrary path of minimal length among all paths joining 
any two distinct $\phi$-equivalent vertices and let $v \neq v'$ be its ends. For any $i
\neq j \!\!\mod n$, let us consider the two paths $\phi^i(\alpha)$ and $\phi^j(\alpha)$,
and the four different situations described in Lemmas~\ref{paths-disjoint/thm} to
\ref{paths-doubled/thm}, which cover all the possibilities, thanks to the
$\phi$-equivalence of $v$ and $v'$.

The situation of Lemma~\ref{paths-doubled/thm} cannot occur, due to the choice of $i$ and
$j$, while any of the three remaining situations may occur. In particular,
$\phi^i(\alpha)$ and $\phi^j(\alpha)$ can have non-empty intersection both as in
Lemma~\ref{paths-diagonal/thm} and as in Lemma~\ref{paths-adjacent/thm}.

Let us assume first that $\phi^i(\alpha)$ and $\phi^j(\alpha)$ are as in 
Lemma~\ref{paths-diagonal/thm}, i.e. they share both ends, with $\phi^i(v) = \phi^j(v')$ 
and $\phi^j(v) = \phi^i(v')$. In particular, $\phi^{j-i}$ switches $v$ and $v'$, and $j - 
i = n/2 \!\!\mod n$, since $\ord_\phi(v) = n$.

The global minimality of $\alpha$ ensures that either $\alpha$ is a single ``diagonal''
edge $D$ with $\phi^{n/2}(D) = \bar D$, or $\alpha$ and $\phi^{n/2}(\alpha)$ only meet at
\pagebreak
their common ends $v$ and $v'$. Moreover, $\alpha$ and $\phi^{n/2}(\alpha)$ are disjoint
from any other $\phi^k(\alpha)$ with $k \neq 0,n/2 \!\!\mod n$, thanks to
Lemma~\ref{paths-disjoint/thm}. If $\alpha$ is a single ``diagonal'' edge $D$, we have
that $\cup_i\, \phi^i(\alpha)$ is as in point 2 of the statement. Otherwise,
$\cup_i\,\phi^i(\alpha)$ consists in the disjoint union $C \sqcup \phi(C) \sqcup \dots
\sqcup \phi^{n/2 - 1}(C)$ of $n/2$ simple cycles (possibly a single one, for $n = 2$),
with $C = \alpha \cup \phi^{n/2}(\alpha)$. Then, arguing as in the proof of
Lemma~\ref{cycle-p/thm}, the length of $\alpha$ (and that of its images as well) may be
reduced to 1, which gives a special case of the situation described in point 1 of the
statement.

At this point, we are left with the case when $\phi^i(\alpha)$ and $\phi^j(\alpha)$ can
only meet as in Lemma~\ref{paths-adjacent/thm}. In this case, denoting by $\phi^i(u)$ with
$u \in \alpha$ the common end of $\delta_{i,1}$ and $\delta_{j,k}$, we have that
$\phi^j(u) \in \delta_{j,k}$ since $\len(\delta_{i,1}) = \len(\delta_{j,k}) \geq
\len(\alpha)/2$. Actually, $\phi^j(u)$ has to coincide with either $\phi^j(v')$ or
$\phi^i(u)$, otherwise the global minimality of $\alpha$ would be contradicted. However,
$\phi^j(u) = \phi^i(u)$ is not possible because this would imply that $\ord_\phi(u) < n$.
Thus, $\phi^j(u) = \phi^j(v')$ and we can conclude that the set $\cup_i\, \phi^i(\alpha)$
has the structure described in point 1 of the statement, by arguing as in the proof of
Lemma~\ref{cycle-p/thm}.
\end{proof}

\begin{lemma}\label{step-2/thm}
Let $\phi: \Gamma \to \Gamma$ be an automorphism of a (possibly disconnected)
uni/trivalent graph $\Gamma$, with $\ord(\phi) = n = 2^m$ and $\ord_\phi(e) = n$ for every
edge $e$ of $\,\Gamma$.\break If $\alpha \subset \Gamma$ is a (simple) path of minimal
length among all paths joining any two distinct terminal edges in a given $\phi$-orbit,
then up to $F$-equivalence we can assume $\cup_i\,\phi^i(\alpha)$ to be one of the
following:
\begin{itemize}
\item[1)] of a disjoint union
$C \sqcup \phi(C) \sqcup \dots \sqcup \phi^{n/\ell - 1}(C)$ of $n/\ell$ simple cycles
(possibly a single one for $\ell = n$), each given by a concatenation of images of
$\alpha$, having the same structure as described in Lemma~\ref{step-p/thm}, with $\alpha$
consisting of either one edge $a$ or two edges $a$ and $b$ (see Figure~\ref{onion-2/fig},
where $\ell$ is a power of $2$, $a_i$ and $b_i$ stand for $\phi^{is}(a)$ and
$\phi^{is}(b)$ respectively, and a similar notation is adopted for the terminal edges
$e_i$'s and the vertices $v_i$'s as well);

\begin{Figure}[htb]{onion-2/fig}
\vskip-6pt
\fig{}{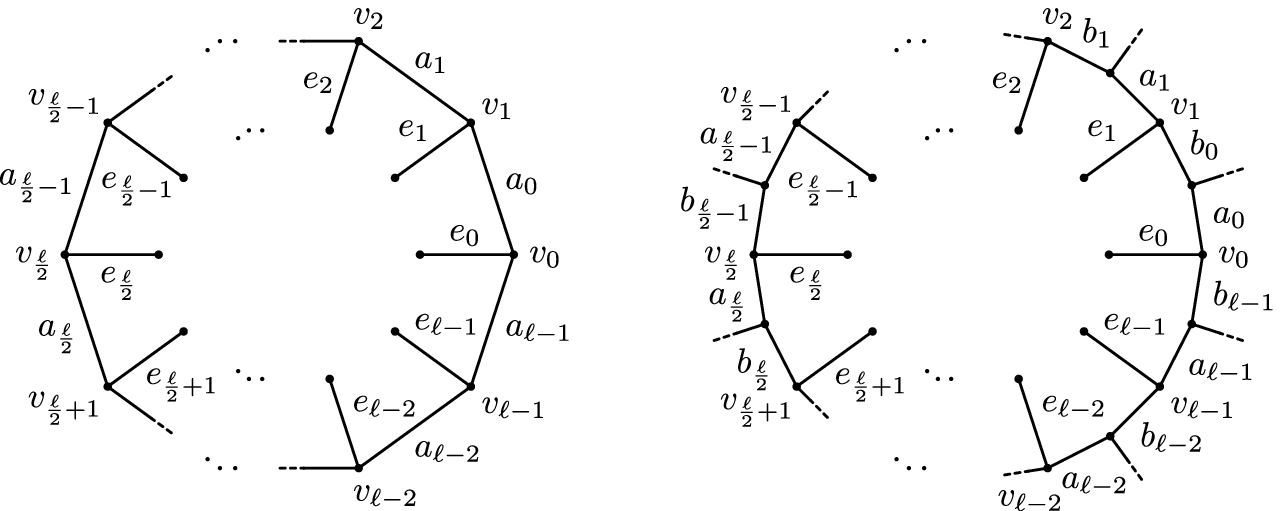}
{}{The form of the cycle $C$ in Lemma~\ref{step-2/thm}.}
\end{Figure}

\item[2)] a disjoint union $D \sqcup \phi(D) \sqcup \dots \sqcup \phi^{n/2\,-\,1}(D)$ of
$n/2$ ``diagonal'' edges (possibly a single one, for $n = 2$), where $D = \alpha$ and
$\phi^{n/2}(D) = \bar D$ (see Figure~\ref{diagonal2/fig}, where the edges $e_0$ and
$e_1$ are switched by $\phi^{n/2}$).
\end{itemize}
\end{lemma}

\begin{Figure}[htb]{diagonal2/fig}
\vskip6pt
\fig{}{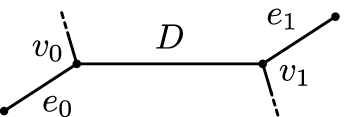}
{}{The ``diagonal'' edge $D$ in Lemma~\ref{step-2/thm}.}
\end{Figure}

\begin{proof}
Let $\alpha$ be a path as in the statement, and let $e \neq e'$ be the terminal edges it
joins. Then the ends of $\alpha$ coincide with the unique trivalent ends $v$ and $v'$ of
$e$ and $e'$ respectively. Notice that $v \neq v'$ since otherwise $\ord_\phi(v)$ would be
less than $n$.

For $i \neq j \!\!\mod n$, the two paths $\phi^i(\alpha)$ and $\phi^j(\alpha)$ can meet
both as in Lemma~\ref{paths-adjacent/thm} and as in Lemma~\ref{paths-diagonal/thm}, but 
not as in Lemma~\ref{paths-doubled/thm}. 

Let us assume first that $\phi^i(\alpha)$ and $\phi^j(\alpha)$ are as in 
Lemma~\ref{paths-diagonal/thm}, i.e. they share both ends, with $\phi^i(v) = \phi^j(v')$ 
and $\phi^j(v) = \phi^i(v')$, and their union $\phi^i(\alpha) \cup \phi^j(\alpha)$ looks 
like in Figure~\ref{paths-diagonal/fig}. In particular, $\phi^{j-i}$ switches $v$ and 
$v'$, and then $j - i = n/2 \mod n$, since $\ord_\phi(v) = n$. Moreover, 
Lemma~\ref{paths-disjoint/thm} ensures that $\phi^i(\alpha) \cup \phi^j(\alpha)$ is 
disjoint from any other path $\phi^k(\alpha)$ with $k \neq i,j \!\!\mod n$.

Hence, $\cup_i\,\phi^i(\alpha)$ is a disjoint union of $n/2$ pairs of ``diagonal'' paths,
each pair being given by $\phi^i(\alpha)$ and $\phi^{i+n/2}(\alpha)$ for some $i = 0,
\dots, n/2-1$. Figure~\ref{onion-2-pf1/fig} shows the two possible forms of
$\phi^i(\alpha) \cup \phi^{i+n/2}(\alpha)$, depending on the fact that $\alpha$ and
$\phi^{n/2}(\alpha)$ start and end with non-trivial common subpaths (left side) or not
(right side). Here, the subscripts denote the images under the corresponding power of
$\phi$, and apart from the edges $e_i$ and $e_{i+n/2}$ each arc represents a path of
edges. We notice that in both cases, if the length of the path $\alpha$ is even, then the
central vertices of the paths $\alpha_i$ and $\alpha_{i+n/2}$ have to be distinct (that
is, they have to belong to non-common subarcs), otherwise the order of such vertices would
be $n/2$.

\begin{Figure}[htb]{onion-2-pf1/fig}
\fig{}{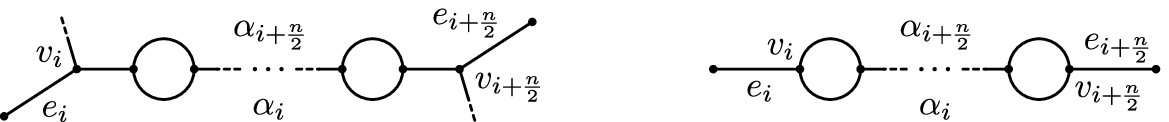}
{}{The starting form of $\phi^i(\alpha) \cup \phi^{i+n/2}(\alpha)$ in the proof of 
   Lemma~\ref{step-2/thm}.}
\end{Figure}

\begin{Figure}[b]{onion-2-pf2/fig}
\fig{}{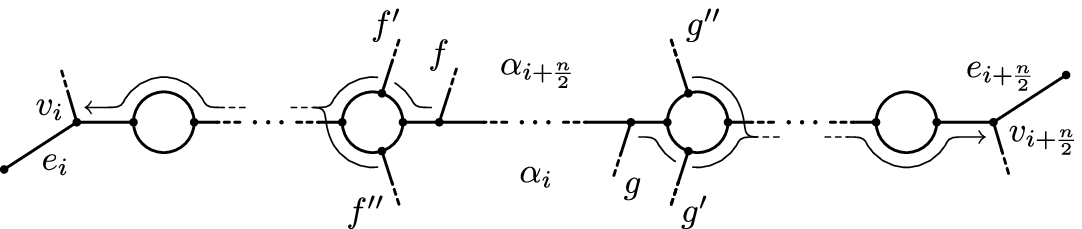}
{}{Simplifying the configuration on the left side of Figure~\ref{onion-2-pf1/fig}.}
\end{Figure}

The configuration on the left side of Figure~\ref{onion-2-pf1/fig} can be simplified by a
sequence of symmetric pairs of $\phi$-invariant slidings as indicated in
Figure~\ref{onion-2-pf2/fig} (cf. Figure~\ref{onion-pf2/fig}), in order to reduce all the
arcs to single edges. After that, all the resulting bigons, except the
central one if they are odd in number, can be modified in pairs as in
Figure~\ref{tripode3/fig} and then slided out of $\phi^i(\alpha) \cup
\phi^{i+n/2}(\alpha)$ as above. Then, in the case of an even number of bigons we end up 
with the situation described in point 1 of the statement. Otherwise, we are left with one 
bigon as shown on the left side of Figure~\ref{onion-2-pf3/fig}, and we can perform the
\pagebreak 
$F$-move described in that figure, to get the configuration on the right side of the same 
Figure~\ref{onion-2-pf3/fig}. This is the special case for $\ell = 2$ of the configuration 
on the right side of Figure~\ref{onion-2/fig} considered in point 1 of the statement.

\begin{Figure}[htb]{onion-2-pf3/fig}
\vskip-6pt
\fig{}{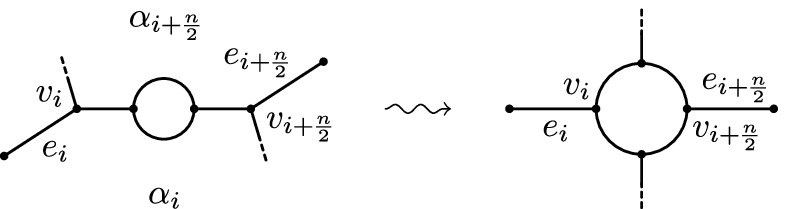}
{}{Simplified version of the configuration on the left side of
   Figure~\ref{onion-2-pf1/fig}.}
\end{Figure}

Now, look at the configuration on the right side of Figure~\ref{onion-2-pf1/fig}. If 
$\len(\alpha) = 1$, that is $\phi^i(\alpha) \cup \phi^{i+n/2}(\alpha)$ consists of a 
single bigon, then we just have the special case for $\ell = 2$ of the configuration on 
the left side of Figure~\ref{onion-2/fig} considered in point 1 of the statement. If 
$\len(\alpha) > 1$, then we can simplify the configuration by the same argument as above, 
but sliding everything to the most external loops adjacent to $e_i$ and $e_{i + n/2}$.
The result is shown on the left side of Figure~\ref{onion-2-pf4/fig}. We can change it 
into the configuration on the right side of the figure, by a $\phi$-invariant $F$-move on 
the central edge, and then to the configuration on the right side of 
Figure~\ref{onion-2-pf3/fig}, by a further $\phi$-invariant sliding (reducing the length 
of $\alpha$ to 2). So, we get once again a special case of the situation described in 
point 1 of the statement.

\begin{Figure}[htb]{onion-2-pf4/fig}
\fig{}{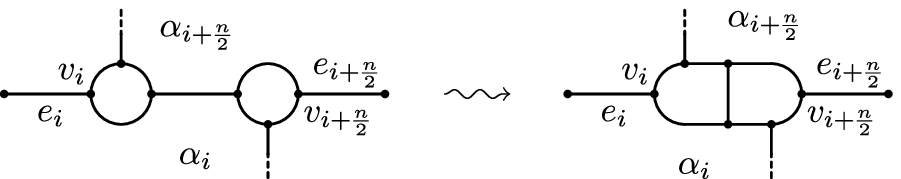}
{}{Simplified version of the configuration on the right side of 
   Figure~\ref{onion-2-pf1/fig}.}
\end{Figure}

At this point, we are left with the case when $\phi^i(\alpha)$ and $\phi^j(\alpha)$ can
only meet as in Lemma~\ref{paths-adjacent/thm}. In this case, then the fact that $n$ is
even ensures that their common end $\phi^j(v) = \phi^i(v')$ cannot be shared by any other
$\phi^h(\alpha)$ and that $\len(\delta_{i,1}) = \len(\delta_{j,k}) > \len(\alpha)/2$ (cf.
Figure \ref{paths-adjacent/fig}). Therefore, the same arguments as in
Lemma~\ref{cycle-p/thm} ensure that the situation can be reduced to the one described in
point 1 of the statement.
\end{proof}

\begin{proposition}\label{onion-2/thm}
Let $\phi: \Gamma \to \Gamma$ be a non-trivial automorphism in $\A_{g,b}$, with
$\ord(\phi) = n = 2^m$. Then, up to $F$-equivalence and composition with elementary
automorphisms, we can reduce $\phi$ to the composition $\tau \circ \sigma$ of two
automorphisms $\sigma, \tau \in \A_{g,b}$ such that $\ord(\sigma) = \ord(\tau) = 2$.
Moreover, $\sigma$ can be assumed to fix an edge and $\tau$ can be assumed to reverse an
invariant edge.
\end{proposition}

\begin{proof}
Given $\phi: \Gamma \to \Gamma$ as in the statement, thanks to Lemma~\ref{reduction/thm}
we can as\-sume that $\ord_\phi(e) = n$ for every edge $e$ of $\,\Gamma$. Then, the same
recursive construction of the proof of Proposition~\ref{onion-p/thm} will provide a
sequence of automorphisms $\phi_i: \Gamma_i \to \Gamma_i$ and a sequence of subgraphs
$\Lambda_i \subset \Gamma_i$ with $i = 0, \dots, r$, satisfying the properties required in
that proof, except for the following facts: 
\begin{itemize}
\item[{\sl a\/})] $\Lambda_1$ is allowed to be a subgraph of the second barycentric
subdivision of $\Gamma_1$ instead that a subgraph of $\Gamma_1$ itself;

\pagebreak

\item[{\sl b\/})] $\ord_{\phi_i}(e) = n$ for every edge $e$ of $\Gamma_i$ and every $i = 
1, \dots, r$;
\item[{\sl c})] $\sigma_i$ fixes exactly two vertices $u_i$ and $\phi_i^{n/2}(u_i)$ in 
$U_i$, for every $i = 0, \dots, r-1$; 
\item[{\sl d})] $\sigma_i$ fixes two edges of $\Gamma_i$, but it is not required to
reverse any invariant edge, for every $2 \leq i \leq r$;
\item[{\sl e})] the $\tau_i$'s are not required to fix or reverse any edge, except for
$\tau_r$ (the last one) that has to reverse an invariant edge.
\end{itemize}
In particular, all the $\phi_i$'s are $F$-equivalent to $\phi$, and all the restrictions
$\phi_{i|\Lambda_i}$ admit a factorization $\tau_i \circ \sigma_i$ into two involutions of
$\Lambda_i$. Hence, $\phi$ turns out to be $F$-equivalent to $\phi_r = \tau_r \circ
\sigma_r$ (being $\Lambda_r = \Gamma_r$), which proves the proposition.

\smallskip

The starting step of the recursion is provided by Lemma~\ref{cycle-2/thm}. If the
situation described in point 1 of that lemma occurs, we define the automorphism $\phi_1:
\Gamma_1 \to \Gamma_1$, the subgraph $\Lambda_1 \subset \Gamma_1$, the two involutions
$\sigma_1,\tau_1: \Lambda_1 \to \Lambda_1$ and the vertex $u_1$, as in the proof of
Proposition~\ref{onion-p/thm}. Otherwise, the situation described in point 2 of
Lemma~\ref{cycle-2/thm} occurs. In this case, there exists a $\phi$-invariant disjoint
union of $n/2$ ``diagonal'' edges $D \sqcup \phi(D) \sqcup \dots \sqcup
\phi^{n/2\,-\,1}(D) \subset \Gamma$, with $\phi^{n/2}(D) = \bar D$. Then, we put $\Gamma_1
= \Gamma_0 = \Gamma$, $\phi_1 = \phi_0 = \phi$ and $\Lambda_1 = D' \sqcup \phi(D') \sqcup
\dots \sqcup \phi^{n/2\,-\,1}(D') \subset \Gamma_1$, where $D' \subset D$ is the start of
the barycenter of $D$ in the second barycentric subdivision of $\Gamma$. Moreover, we
denote by $u$ any one of the two ends of $D'$, and we define $\sigma_1$ and $\tau_1$ by
putting $\sigma_1(\phi^j(u)) = \phi^{-j}(u)$ and $\tau_1(\phi^j(u)) = \phi^{1-j}(u)$ for
$j = 0, \dots, n-1$.

In order to conclude the starting step, it is enough to observe that: $\ord_{\phi_1}(e) =
n$ for every edge $e$ of $\Gamma_1$; $\sigma_1$ and $\tau_1$ are well-defined involutive
automophisms of $\Lambda_1$, since they preserve the adjacency of vertices, being
$\phi^j(u)$ adjacent to $\phi^k(u)$ if and only if $j = k \!\!\mod n/2$; $\tau_1 \circ
\sigma_1$ coincides with the restriction $\phi_{1|\Lambda_1}$; $\Delta_1 = \Cl(\Gamma_1 -
\Lambda_1)$ is a uni/trivalent graph; $U_1 = \Delta_1 \cap \Lambda_1$ consists of the
$\phi_1$-orbit of the univalent vertex $u$ of $\Delta_1$ and it has cardinality $n$; $u_1
= u$ and $\phi^{n/2} (u_1)$ are the only vertices of $U_1$ fixed by $\sigma_1$.

Now, to realize the recursive step of the construction, assume we are given $\phi_{i-1}:
\Gamma_{i-1} \to \Gamma_{i-1}$, $u_{i-1} \in U_{i-1} \subset \Lambda_{i-1} \varsubsetneq
\Gamma_{i-1}$ and $\sigma_{i-1},\tau_{i-1}: \Lambda_{i-1} \to \Lambda_{i-1}$ with the
properties required in the proof of Proposition~\ref{onion-p/thm} for $i-1 < r$,
integrated by the points {\sl a\/} to {\sl e\/} said above. In particular, the restriction 
of $\phi_{i-1}$ to the uni/tri\-valent graph $\Delta_{i-1} = \Cl(\Gamma_{i-1} - 
\Lambda_{i-1})$ is an automorphism of order $n$, and the order of each edge of 
$\Delta_{i-1}$ with respect to it is $n$ as well.

If two of the terminal edges of $\Delta_{i-1}$ ending at vertices in $U_{i-1}$ coincide, 
then there are $n/2$ such terminal edges and each of them is a ``diagonal'' edge reversed 
by $\phi_{i-1}^{n/2}$. In this case, denoting by $D$ the ``diagonal'' edge between 
$u_{i-1}$ and $\phi_{i-1}^{n/2}(u_{i-1})$, we put $\Gamma_i = \Gamma_{i-1}$, $\phi_i = 
\phi_{i-1}$ and $\Lambda_i = \Lambda_{i-1} \cup_j\,\phi^j_{i-1}(D)$, where the last 
equality is due to the connectedness of $\Gamma_i$. Moreover, we define $\sigma_i$ and 
$\tau_i$ to be the unique extensions of $\sigma_{i-1}$ and $\tau_{i-1}$ to $\Gamma_i$. 
These can be easily seen to exist and to verify all the required properties for $i = r$,
except for the fact that $\tau_r$ reverses some invariant edge, which\break will be proved 
at the end of the proof. Hence, this step terminates the recursion.

Then, let the terminal edges of $\Delta_{i-1}$ ending at vertices in $U_{i-1}$ be all
distinct.

If $\Delta_{i-1}$ does not contain any path joining two different such terminal edges, 
then we can terminate the recursion with the $i$-th step, in the same way as in the proof 
of Proposition~\ref{onion-p/thm}. Once again, we postpone at the end of the proof the 
verification that $\tau_r$ reverses some invariant edge.

Otherwise, if a path $\alpha \subset \Delta_{i-1}$ exists joining different terminal edges
of $\Delta_{i-1}$ ending at vertices in $U_{i-1}$, then we can choose it to be minimal
(hence simple) and apply Lemma \ref{step-2/thm} to the restriction of $\phi_{i-1}$ to
$\Delta_{i-1}$ and to such a minimal $\alpha$, in order to get the structure described in
that statement for $\cup_j\, \phi_{i-1}^j(\alpha) \subset \Delta_{i-1}$ up to
$F$-equivalence. Such an $F$-equivalence only involves internal edges of $\Delta_{i-1}$,
hence it does not change the set of free ends $U_{i-1} \subset \Delta_{i-1}$ and the
restriction of $\phi_{i-1|\Delta_{i-1}}$. Therefore, it can be extended to an
$F$-equivalence of the whole $\phi_{i-1}$ on $\Gamma_{i-1}$, which leaves $\Lambda_{i-1}$ 
and the restriction $\phi_{i-1|\Lambda_{i-1}}$ unchanged.

As a result, we get a new uni/trivalent graph $\Gamma_i$, such that $\Lambda_{i-1} \subset
\Gamma_i$ and a new automorphism $\phi_i: \Gamma_i \to \Gamma_i$, which is $F$-equivalent
to $\phi_{i-1}$ and coincides with $\phi_{i-1}$ on $\Lambda_{i-1}$. We also get a new
minimal path $\alpha \subset \Cl(\Gamma_i - \Lambda_{i-1})$, which joins different
$\phi_i$-equivalent terminal edges of $\Cl(\Gamma_i - \Lambda_{i-1})$ ending at vertices
in $U_{i-1}$, such that $\cup_j\, \phi_i^j(\alpha)$ itself (no more up to
$F$-equivalence) has the structure stated in Lemma~\ref{step-2/thm}. Then, we proceed by 
distinguishing the different situations described in points 1 and 2 of that lemma.

If the situation of point~1 occurs, then the $i$-th step is the same as in the proof of
Proposition~\ref{onion-p/thm}, but now $\sigma_i$ fixes two edges and $\tau_i$ reverses
two invariant edges.

In the case when the situation of point~2 occurs, we denote by $e$ the terminal edge of
$\Cl(\Gamma_i - \Lambda_{i-1})$ whose free end is $u_{i-1}$. Then, we choose $u_i$ to be
the other end of $e$ and put $\Lambda_i = \Lambda_{i-1} \cup_j\, \phi_i^j(e) \cup_j\,
\phi_i^j(D)$, in such a way that $\Delta_i = \Cl(\Gamma_i - \Lambda_i)$ is a uni/trivalent
graph, which meets $\Lambda_i$ at the $\phi_i$-orbit $U_i$ of $u_i$. Moreover, we define
$\sigma_i$ and $\tau_i$ to be the unique automorphisms of $\Lambda_i$ extending
$\sigma_{i-1}$ and $\tau_{i-1}$ respectively. To see that such extensions exist, we first
define $\sigma_i(\phi_i^j)(u_i) = \phi_i^{-j}(u_i)$ and $\tau_i(\phi_i^j)(u_i) =
\phi_i^{1-j}(u_i)$, for every $j = 0, \dots, n-1$, and then we verify that these
definitions, together with the previous definitions of $\sigma_{i-1}$ and $\tau_{i-1}$,
preserve adjacency of vertices in $\Lambda_i$. Notice that $\tau_i \circ \sigma_i$
trivially coincides with $\phi_{i|\Lambda_i}$. Furthermore, $\sigma_i$ fixes only the
vertices $u_i$ and $\phi_i^{n/2}(u_i)$ in the $\phi_i$-orbit $U_i$, and it fixes the edges
$e$ and $\phi^{n/2}(e)$. On the contrary, $\tau_i$ neither fixes nor reverses any edge in
the $\phi_i$-orbits of $e$ and $D$. This conclude the recursive step.

Finally, we observe that the situation of point 1 of Lemma~\ref{step-2/thm} occurs in at 
least one of the recursive steps due to the connectedness of $\Gamma$, and from that 
step on all the $\tau_i$ reverse some invariant edge. In particular, this is true for 
$\tau_r$.
\end{proof}

At this point, taking into account Propositions~\ref{onion-p/thm}, \ref{onion-3/thm} and
\ref{onion-2/thm}, we are left to show that any automorphism in $\A_{g,b}$ belongs to
$\E_{g,b}$, in the special case when it has order 2 and fixes or reverses some edge. This
is the content of our last proposition.

\begin{proposition}\label{automorphism-2/thm}
Let $\phi:\Gamma \to \Gamma$ be an automorphism in $\A_{g,b}$ of order $2$, which either
fixes an edge or reverses an invariant edge. Then, up to $F$-equivalence and composition
with elementary switches, we can reduce $\phi$ to the identity.
\end{proposition}

\begin{proof}
If $\phi$ fixes an edge $e$, then Lemma~\ref{reduction/thm} ensures that up to
$F$-equivalence and composition with elementary automorphisms, $\phi$ can be reduced to an
automorphism $\psi: \Gamma' \to \Gamma'$ such that $\ord(\psi) = 1$, that is $\psi$
reduces to the identity. On the other hand, if $\phi$ reverses an invariant edge $e$, we 
can perform a $\phi$-invariant $F$-move $F_{e,e'}$, in order to get a new automorphism
$\phi': \Gamma' \to \Gamma'$, which fixes the edge $e'$ of $\Gamma'$.
\end{proof}

\newpage

\bibliographystyle{amsplain}

\end{document}